\documentclass[11pt]{article}
\usepackage[intlimits]{amsmath}
\usepackage{amsfonts, amssymb, amsthm}
\usepackage{epsfig}
\usepackage{color}

\newtheorem{thm}{Theorem}[section]

\newtheorem{cor}[thm]{Corollary}
\newtheorem{lemma}[thm]{Lemma}
\newtheorem{defn}[thm]{Definition}

\newtheorem{preremark}[thm]{Remark}
\newenvironment{remark}{\begin{preremark}\rm}{\medskip \end{preremark}}
\numberwithin{equation}{section}

\newcommand\res{\hbox{ {\vrule height .22cm}{\leaders\hrule\hskip.2cm} } }

\newcommand{\R}{\mathbb R}

\newcommand{\dist} {\mathrm{dist}}

\def\Xint#1{\mathchoice
                 {\XXint\displaystyle\textstyle{#1}}%
                 {\XXint\textstyle\scriptstyle{#1}}%
                 {\XXint\scriptstyle\scriptscriptstyle{#1}}%
                 {\XXint\scriptscriptstyle\scriptscriptstyle{#1}}%
                 \!\int}
\def\XXint#1#2#3{{\setbox0=\hbox{$#1{#2#3}{\int}$}
       \vcenter{\hbox{$#2#3$}}\kern-.5\wd0}}

\newcommand{\meanbar}[1]{%
\setbox0 = \hbox{$#1 \int$}
\hbox to 0pt{%
\thinspace
\hskip 0.1\wd0
\raise 0.5\ht0
\hbox{%
\lower 0.5\dp0
\hbox{\rule{0.8\wd0}{2\linethickness}}
}%
\hss
}%
}
\newcommand{\palette}[1]{%
\mathchoice{#1 \displaystyle}%
{#1 \textstyle}%
{#1 \scriptstyle}%
{#1 \scriptscriptstyle}%
}
\newcommand{\mean}{\palette \meanbar}
\newcommand{\mint}{\mean \int}
\begin{document}
\title{Divergence form operators in Reifenberg flat domains.}

\author{Emmanouil Milakis\footnote{email: milakis@math.washington.edu}\hspace{0.2em} and Tatiana Toro\footnote{email: toro@math.washington.edu}\footnote{The second author was partially supported by NSF grant DMS-0600915.}\\
University of Washington\\
Department of Mathematics\\
P.O Box 354350 Seattle, WA 98195, USA }
\date{}
\maketitle

\begin{abstract} We study the boundary regularity of solutions of elliptic operators in divergence form with $C^{0,\alpha}$ coefficients or operators which are small perturbations of the Laplacian in non-smooth domains. We show that, as in the case of the Laplacian, there exists a close relationship between the regularity of the corresponding elliptic measure and the geometry of the domain.
\end{abstract}

\section{Introduction}
\renewcommand{\thesection}{\arabic{section}}

The basic aim of this paper is to study the relationship between
the elliptic measure of operators in divergence form with $C^{0,\alpha}$ coefficients or operators which are small perturbations of the Laplacian, and
the geometry of the boundary of the domain.
We concentrate on domains whose boundary is {\it locally flat},
where this notion will be understood in a weak sense.
Let $\Omega\subset\R^{n+1}$ be an open set. Loosely speaking we say that
$\partial\Omega$ is locally flat if locally it can be well approximated
by affine spaces. In particular,  such domains
are non-tangentially accessible and therefore their
elliptic measure $\omega$ is doubling
(see \cite{jk82}, \cite{k1}).

We prove that if $\partial\Omega$ is well approximated
by $n$-planes in the Hausdorff distance sense then the doubling constant
of the elliptic measure of divergence form operators with H\"older coefficients, $\omega$,
asymptotically approaches the doubling constant of the
$n$-dimensional Lebesgue measure. If moreover the unit normal vector to
$\partial\Omega$ has small (resp. vanishing) mean oscillation, then the logarithm of the corresponding elliptic kernel has small (resp. vanishing) mean oscillation.

In \cite{d1}, Dahlberg showed that if $\Omega$ is Lipschitz domain
then the harmonic measure and the surface measure are mutually
absolutely continuous. In addition the Poisson kernel
is a $B_2$ weight with respect to the surface measure to the boundary, which
implies that
the logarithm of the Poisson kernel is a function of bounded mean
oscillation with respect to the surface measure on $\partial
\Omega$ (i.e it is in $BMO(\partial\Omega$)). Jerison and Kenig \cite{jk82}, showed that if $\Omega$ is
a $C^1$ domain then the logarithm of the Poisson kernel is in
$VMO(\partial \Omega)$. In the case when the boundary is
locally flat, Kenig and Toro gave a full description of the relation between the harmonic measure of
the domain and the geometry of its boundary (see \cite{kt1}). More
precisely if $\Omega\subset \R^{n+1}$ is an open connected
Reifenberg flat domain with vanishing constant then the harmonic
measure behaves asymptotically like the Euclidean measure. If
$\Omega$ is a chord arc domain with vanishing constant then the
logarithm of the Poisson kernel has vanishing mean oscillation (i.e. it is in $VMO{(\partial \Omega)}$).

The regularity of the elliptic kernel for divergence form operators which are perturbations of the Laplacian
on Lipschitz domains has been studied by several authors.
Dahlberg \cite{d}, showed that if the difference between the coefficients of an elliptic, divergence form operator
$L$, and the Laplacian satisfies a Carleson condition with vanishing trace then the corresponding elliptic
kernel is a $B_2$ weight with respect to surface measure.
In \cite{fkp}, Fefferman, Kenig and Pipher studied the case when the same Carleson condition is satisfied but without the smallness assumption. In that case, the elliptic measure of $L$  is an $A_\infty$ weight with respect to surface measure.
In \cite{esc1} Escauriaza proved that on a $C^1$ domain if the difference between the coefficients of
$L$ and the Laplacian satisfies a Carleson condition with vanishing trace then the logarithm  of the elliptic
kernel is in $VMO(\partial\Omega)$.

In this paper we extend the results of \cite{kt1} to more general uniformly elliptic operators in divergence form.
In section 2 we present the preliminaries, define the two classes of operators we intend to study and
state our main results.
In section \ref{optimalDoubling} we prove that the elliptic measure of a divergence form elliptic operator with H\"older coefficients on a Reifenberg flat domain with vanishing constant is asymptotically optimally doubling.  The proofs in
this section follow the arguments presented in \cite{kt1}.
In section \ref{chord} we show that, in a chord arc domain with vanishing constant,  the logarithm of the corresponding elliptic kernel is in $VMO$.  In section \ref{chord}, we also
extend some of the results  in \cite{fkp} to chord arc domains with small constant.
A natural question is whether Escauriaza's  result (see \cite{esc1})
generalizes  to chord arc domains with vanishing constant. We expect this to be the case.

\section{Preliminaries and results.}
\renewcommand{\thesection}{\arabic{section}}

In this section we recall some definitions and state our main results. First we introduce the class of Reifenberg flat domains, which are domains whose boundary can be well approximated by planes.
In particular Lipschitz domains with small constant are Reifenberg flat.
\begin{defn}
Let $\Omega \subset \R^{n+1}$ be a bounded domain, we say that $\partial\Omega$
separates $\R^{n+1}$ if there exist $\delta>0$, and $R>0$ such
that for each $Q\in\partial\Omega$, there exist an $n$-dimensional
plane $\mathcal{L}(Q,R)$ containing $Q$ and a choice of unit
normal vector to $\mathcal{L}(Q,R)$, $n_{Q,R}$ satisfying
\begin{equation}\label{1.1}
{T}^{+}(Q,R)=\{X=(x,t)=x+t n_{Q,R}\in B(Q,R): x\in
{\mathcal L}(Q,R),\ t>2\delta R\}\subset  \Omega,
\end{equation}
and
\begin{equation}\label{1.2}
{T}^{-}(Q,R)=\{X=(x,t)=x+t n_{Q,R}\in B(Q,R): x\in
{\mathcal L}(Q,R),\ t<-2\delta R\}\subset  \Omega^{c}.
\end{equation}
\end{defn}
Here $B(Q,R)$ denotes the $(n+1)$-dimensional ball of radius
$R$ and center $Q$.
\begin{defn}
Let $\Omega\subset  \R^{n+1}$, $\delta> 0$, $R>0$. We say that
$\Omega$ is a $(\delta, R)$-Reifenberg flat domain if
$\partial\Omega$ separates $\R^{n+1}$, and for each
$Q\in\partial\Omega$, and for every $r\in (0, R]$ there
exists an $n$-dimensional plane $\mathcal{L}(Q,R)$ containing $Q$ such
that
\begin{equation}\label{1.3}
\frac{1}{r}D[\partial\Omega\cap B(Q,r), \mathcal{L}(Q,r)\cap B(Q,r)]\leq \delta.
\end{equation}
where $D$
denotes the Hausdorff distance.
\end{defn}

We denote by
\begin{equation}\label{1.4}
\theta(r)=\sup_{Q\in \partial\Omega}\inf_{\mathcal{L}}\left\{\frac{1}{r} D[\partial\Omega\cap B(Q,r), \mathcal{L}\cap B(Q,r)]\right\},
\end{equation}
where the infimum is taken over all $n$-planes containing $Q$.

\begin{defn}
Let $\Omega \subset \R^{n+1}$, we say that $\Omega$ is a
Reifenberg flat domain with vanishing constant if it is $(\delta,
R)$-Reifenberg flat for some $\delta>0$ and $R>0$, and
\begin{equation}\label{1.5}
\limsup_{r \rightarrow 0}\theta(r)=0.
\end{equation}
\end{defn}

Note that definitions \ref{1.1} and \ref{1.2} are only significant
for $\delta>0$ small. Thus when talking about $(\delta, R)$-Reifenberg flat domains we assume that $\delta$ is small enough. In particular, we assume that $\delta$ is small enough so that if $\Omega$ is a
$(\delta, R)$ Reifenberg flat domain it is also an NTA domain (see \cite{kt1}).

\begin{defn}
Let $\Omega \subset \R^{n+1}$. We say that
$\Omega$ is a chord arc domain (CAD) if $\Omega$ is an NTA
set of locally finite perimeter such that there exists $C>1$ so that for $r\in (0,{\rm{diam}}\ \Omega)$ and $Q\in\partial \Omega$
\begin{equation}\label{1.7A}
C^{-1}r^n\leq\sigma(B(Q,r))\leq Cr^{n}.
\end{equation}
Here $\sigma=\mathcal{H}^{n}\res \partial\Omega$ and $\mathcal{H}^{n}$ denotes the $n$-dimensional Hausdorff measure.
\end{defn}

\begin{defn}
Let $\Omega \subset \R^{n+1}$, $\delta> 0$ and $R>0$. We say that
$\Omega$ is a $(\delta,R)$-chord arc domain (CAD) if $\Omega$ is a
set of locally finite perimeter such that
\begin{equation}\label{1.6}
\sup_{0<r\leq R}\theta(r)\leq \delta
\end{equation}
and
\begin{equation}\label{1.7}
\sigma( B(Q,r))\leq (1+\delta)\omega_{n}r^{n}\ \ \forall
Q\in\partial\Omega\ \ {\rm{and}} \ \forall r\in (0,R].
\end{equation}
Here $\omega_{n}$ is the volume of the $n$-dimensional unit
ball in $\R^{n}$.
\end{defn}

\begin{defn}
Let $\Omega \subset \R^{n+1}$, we say that $\Omega$ is a chord arc
domain with vanishing constant if it is a $(\delta, R)$-CAD for
some $\delta>0$ and $R>0$,
\begin{equation}\label{1.8}
{\lim\sup}_{r\rightarrow 0}\theta(r)=0
\end{equation}
and
\begin{equation}\label{1.9}
\lim_{r\rightarrow 0}\sup_{Q\in\partial \Omega}\frac{\sigma(B(Q,r))}{\omega_{n}r^{n}} =1.
\end{equation}
\end{defn}

For the purpose of this paper we assume that $\Omega\subset \R^{n+1}$ is a bounded
domain.  We consider elliptic
operators $L$ of the form
\begin{equation}\label{div-tt}
Lu=\textrm{div}(A(X)\nabla u)
\end{equation}
defined in the domain $\Omega$ with symmetric coefficient matrix $A(X)=(a_{ij}(X))$ and such that there
are $\lambda, \Lambda>0$ satisfying
\begin{equation}\label{ellipticity}
\lambda |\xi|^2\leq \sum_{i,j=1}^{n+1}a_{ij}(X)\xi_i\xi_j\leq \Lambda |\xi|^2
\end{equation}
for all $X \in \Omega$ and $\xi \in \R^{n+1}$.

We say that a function $u$ in $\Omega$ is a solution to $Lu=0$ in $\Omega$ provided that $u\in W_{\rm{loc}}^{1,2}(\Omega)$ and for all $\phi\in C^{\infty}_c(\Omega)$
$$\int_{\Omega}\langle A(x)\nabla u,\nabla\phi\rangle dx =0.$$

A domain $\Omega$ is called regular for the operator $L$, if for
every $g\in C(\partial \Omega)$, the generalized solution of the
classical Dirichlet problem  with boundary data $g$ is a function $u\in C(\overline{\Omega})$.

\begin{defn}
Let $\Omega$ be a regular domain for $L$ as above and $g\in C(\partial \Omega)$. For $X\in
\Omega$ consider the linear functional $g \rightarrow u(X)$ on
$C(\partial \Omega)$, where $u$ is the generalized solution of the
classical Dirichlet problem with boundary data $g$. By the Riesz representation theorem,
there exists a family of regular Borel probability measures
$\{\omega^X_L\}_{X\in\Omega}$ such that
$$u(X)=\int_{\partial\Omega}g(Q)d\omega^X_L(Q).$$
For $X\in \Omega$, $\omega^X_L$ is called the $L-$elliptic measure of $\Omega$ with pole $X$. When no
confusion arises, we will omit the reference to $L$ and simply
called it as the elliptic measure.
\end{defn}

To state our results  we introduce two classes of operators.

We say that elliptic operator $L\in \mathcal{L}(\lambda, \Lambda,\alpha)$
if it satisfies (\ref{div-tt}), (\ref{ellipticity}) and the modulus of continuity of the corresponding matrix is given, up to the boundary, by
\begin{equation}\label{modulus}
w(r)=\sup_{|X-Y|\leq r}|A(X)-A(Y)|\leq c_0 r^{\alpha}
\end{equation}
for some $\alpha \in (0,1]$, that is $A\in C^{\alpha}(\overline{\Omega})$. Without loss of generality we assume that $A$ is defined in $\R^{n+1}$ since $A$ can be extended to a new matrix in the following way. If we start with $A\in C^{\alpha}(\overline{\Omega})$ then there exists an open set $U$ such that $\overline{\Omega}\subset U$ and $A\in C^{\alpha}(U)$. Consider now a smooth function $\phi \in C^{\infty}_{c}(\R^{n+1})$ which is equal to $1$ in $\overline{\Omega}$ and $0$ outside $U$. We then extend $A$ to $B=\phi A+(1-\phi)I$ in $\overline{\Omega}$ which gives that $B\in C^{\alpha}(\R^{n+1})$ and $B=A$ in $\overline{\Omega}$.

An elliptic operator $Lu={\rm{div}}(A(X)\nabla u)$ defined on a chord arc domain $\Omega\subset \R^{n+1}$ is a perturbation of the Laplacian for the purposes of this paper
if the deviation function
\begin{equation}\label{eqn:tt-a}
a(X)=\sup\{|\rm{Id}-A(Y)|: Y\in B(X,\delta(X)/2)\}
\end{equation}
where
$\delta(X)$ is the distance of $X$ to $\partial \Omega$, satisfies the following Carleson measure property: there exists $C>0$ such that
\begin{equation}\label{normfkp}
\sup_{0<r<\rm{diam}\Omega}\sup_{Q\in\partial \Omega}
\bigg\{\frac{1}{\sigma(B(Q,r))}\int_{B(Q,r)\cap \Omega}\frac{a^2(X)}{\delta(X)}dX\bigg\}\leq C,
\end{equation}
where $\sigma=\mathcal{H}^{n}\res \partial\Omega$.
Note that in this case $L=\Delta$ on $\partial \Omega$ and therefore by letting $L=\Delta$ in $\Omega^c$
we may assume that $L$ is an elliptic operator in $\R^{n+1}$.

We now state some of our results:

\begin{thm}
Let $\Omega\subset \R^{n+1}$ be a Reifenberg flat domain with
vanishing constant, let $L\in \mathcal{L}(\lambda, \Lambda,\alpha)$ and let $\omega$ be its elliptic measure. Then
for all $\tau\in (0,1)$,
$$
\lim_{\rho\rightarrow 0} \inf_{Q\in \partial \Omega}
\frac{\omega(B(Q, \tau\rho))} {\omega(B(Q, \rho))}=
\lim_{\rho\rightarrow 0} \sup_{Q\in \partial \Omega}
\frac{\omega(B(Q,\tau\rho))} {\omega(B(Q,\rho))}=
\tau^{n}.
$$
\end{thm}

In section \ref{chord} we show that if $\Omega$ is a chord arc domain with vanishing constant and
$L\in \mathcal{L}(\lambda, \Lambda,\alpha)$ then $\omega\in A_\infty(d\sigma)$. Furthermore we obtain the following results.

\begin{thm}
Given $\varepsilon>0$, and $\theta>0$ there exists $\delta>0$ such
that if $L\in \mathcal{L}(\lambda, \Lambda,\alpha)$ and
$\Omega\subset \R^{n+1}$ is a $(\delta, R)$-CAD there
exists $r_0>0$, so that for any
$Q \in \partial\Omega$ and $r < r_0$, if
$k(Q)= \frac{d\omega}{d\sigma}(Q)$ denotes the elliptic kernel of $L$, then
$$
\bigg({\Xint-}_{B(Q,r)}k^{1+\beta}d\sigma\bigg)^\frac{1}{(1+\beta)}\leq
(1+\varepsilon){\Xint-}_{B(Q,r)}k d\sigma,
$$
for any $\beta\in (0,1/\theta)$.
\end{thm}

\begin{thm}
Let $\Omega\subset\R^{n+1}$ be a chord arc domain with vanishing
constant. Assume that $L\in \mathcal{L}(\lambda, \Lambda,\alpha)$. Then $\log k\in VMO(\partial\Omega)$.
\end{thm}

We now recall some of the results concerning the regularity of the elliptic measure of perturbation operators
in Lipschitz domains. The results in the literature are more general than those quoted below.

\begin{thm}\label{dahlberg1986}\cite{d}
Let $\Omega=B(0,1)$. If $a$ is as in (\ref{eqn:tt-a}),
\begin{equation}\label{carlesonnorm}
h(Q,r)=\bigg\{\frac{1}{\sigma(B(Q,r))}\int_{B(Q,r)\cap \Omega}\frac{a^2(X)}{\delta(X)}dX\bigg\}
\end{equation}
and
$$
\lim_{r \rightarrow 0}\sup_{|Q|=1}h(Q,r)=0.
$$
Then the elliptic kernel of $L$, $k=d\omega/d\sigma\in B_q(d\sigma)$ for all $q>1$.
\end{thm}

In \cite{feff}, Fefferman made the first step toward removing the smallness condition of $h(Q,r)$ in Theorem \ref{dahlberg1986} by defining an appropriate quantity $A(Q)$.

\begin{thm}\cite{feff}\label{fefferman1989}
Let $\Omega= B(0,1)$. Let $\Gamma(Q)$ denote a non-tangential cone with vertex $Q$ and
$$A(Q)=\bigg(\int_{\Gamma(Q)}\frac{a^2(X)}{\delta^n(X)}dX\bigg)^{1/2},$$
where $a$ is as in (\ref{eqn:tt-a}).
If $\|A\|_{L^{\infty}}\leq C$ then $\omega\in A_{\infty}(d\sigma)$.
\end{thm}

The main results in \cite{d} and in \cite{feff} are proved using a differential inequality for a family of harmonic measures introduced by Dahlberg. In \cite{fkp}, Fefferman, Kenig and Pipher presented a new direct proof of these results without the use of this differential inequality.

\begin{thm}\label{fkp1991}\cite{fkp}
Let $\Omega$ be a Lipschitz domain. Let $L$ be such that (\ref{normfkp})  holds then
 $\omega\in A_\infty(d\sigma)$.
\end{thm}

In this paper we generalize Theorem \ref{fkp1991} to chord arc domains with small constant.

\begin{thm}\label{mainpertu}
Let $\Omega$ be a  chord arc domain. Let $L$ be such that (\ref{normfkp}) holds. There exists $\delta(n)>0$ such that if $\Omega \subset \R^{n+1}$ is a $(\delta,R)-$CAD with $0<\delta\leq\delta(n)$
then $\omega\in A_\infty(d\sigma)$.
\end{thm}

The various constants that will appear in the sequel may vary from formula to formula, although for simplicity we use the same letter(s). If we do not give
any explicit dependence for a constant, we mean that it depends
only on the usual parameters such as ellipticity constants, NTA constants and character of the domain and dimension. Moreover throughout the paper we shall use the notation $a\lesssim b$ to mean that there is a constant $c>0$ such that $ca\leq b$. Similarly $a\simeq b$ means that $a\lesssim b$ and $b\lesssim a$.

Next we recall the main theorems about the boundary behavior of
$L-$elliptic functions in non-tangentially accessible (NTA)
domains for uniformly elliptic divergence form operators $L$ with bounded measurable coefficients. We refer the reader to \cite{k1} for the definitions and
more details regarding elliptic operators of divergence form
defined in NTA domains.

\begin{lemma}\label{lem2.1}
Let $\Omega$ be an NTA domain. If $Lu=0$ in $\Omega\cap B(Q,2r)$ with $0<2r<R$, $u\geq 0$ and vanishes continuously
on $\partial\Omega\cap B(Q,2r)$ then there exists $\beta>0$ such that
for all $Q\in\partial\Omega$ and for $X\in \Omega\cap B(Q,r)$,
$$
u(X)\leq C \bigg(\frac{|X-Q|}{r}\bigg)^{\beta} \sup\{u(Y):
Y\in\partial B(Q,2r)\cap\Omega\}.
$$
\end{lemma}

\begin{lemma}\label{lem2.2}
Let $\Omega$ be an NTA domain, $Q\in\partial\Omega$, and $0<2r<R$.
If $u\geq 0$, $Lu=0$ in $\Omega$ and $u$ vanishes
continuously on $\partial\Omega\cap B(Q,2r)$, then
$$
u(Y)\leq C u(A(Q,r)),
$$
for all $Y\in B(Q,r)\cap\Omega$. Here $C$ only depends on the NTA
constants.
\end{lemma}

\begin{lemma}\label{lem2.3}
Let $\Omega$ be an NTA domain, $Q\in\partial\Omega$, $0<2r<R$, and
$X\in\Omega\backslash B(Q,2r)$. Then
$$
C^{-1}<\frac{\omega^{X}(B(Q,r))}{r^{n-1}|G(A(Q,r),X)|}<C,
$$
where $G(A(Q,r),X)$ is the $L-$Green function of $\Omega$ with pole $X$, and $\omega^X$ is the corresponding elliptic measure.
\end{lemma}

\begin{lemma}\label{lem2.4}
Let $\Omega$ be an NTA domain with constants $M>1$ and $R>0$,
$Q\in\partial\Omega$, $0<2r<R$, and $X\in\Omega\backslash
B(Q,2Mr)$. Then for $s\in [0,r]$
$$
\omega^{X}(B(Q,2s))\le C \omega^{X}(B(Q,s)),
$$
where $C$ only depends on the NTA constants of $\Omega$.
\end{lemma}

\begin{lemma}\label{lem2.5}
Let $\Omega$ be an NTA domain, and $0<Mr<R$. Suppose that $u,v$ vanish continuously on
$\partial\Omega\cap B(Q,Mr)$ for some $Q\in\partial\Omega$, $u,v\geq 0$ and
$Lu=Lv=0$ in $\Omega$. Then there exists a
constant $C>1$ (only depending on the NTA constants) such that for
all $X\in B(Q,r)\cap\Omega$,
$$
C^{-1}\frac{u(A(Q,r))}{v(A(Q,r))}\leq \frac{u(X)}{v(X)}\leq
C\frac{u(A(Q,r))}{v(A(Q,r))}.
$$
\end{lemma}

\begin{thm}\label{Thm2.3}
Let $\Omega$ be an NTA domain. There exists a number $\gamma\in(0,1)$,
such that for all $Q\in\partial\Omega$, $0<2r<R$, and all $u,v\geq 0$ satisfying $Lu=Lv=0$ in $\Omega\cap B(Q,2r)$
and which vanish continuously on $\partial\Omega\cap B(Q,2r)$, the function $\frac{u(X)}{v(X)}$
is H\"older continuous of order $\gamma$ on $\overline{\Omega}\cap
\overline{B(Q,r)}$. In particular, for every $Q\in
\partial\Omega$, $\lim_{X\rightarrow Q}\frac{u(X)}{v(X)}$ exists, and for
$X,Y\in \Omega\cap B(Q,r)$,
$$
\bigg|\frac{u(X)}{v(X)}-\frac{u(Y)}{v(Y)} \bigg|\leq C
\frac{u(A(Q,r))}{v(A(Q,r))}\bigg(\frac{|X-Y|}{r}\bigg)^{\gamma}.
$$
\end{thm}

We finish this section by recalling a result concerning the
regularity of elliptic measure on Lipschitz domains,
as well as some doubling properties of the elliptic measure of
a cylinder. Let $H\subset \R^{n+1}$ be an open half space, for
$M>1$, $s>0$, and $Q_{0}\in\partial H = \mathcal{L}$ we denote by
$$
\mathcal{C}^{+}(Q_{0},Ms)=\{ (x,t)\in\R^{n+1}: x\in \partial H;
|x-Q_{0}|\leq \frac{Ms}{\sqrt{n+1}}, \ |t|\leq
\frac{Ms}{\sqrt{n+1}}\}\cap H,
$$
the cylinder with basis $B(Q_{0}, {{Ms}/{\sqrt{n+1}}})\cap \partial H$ and
height ${{Ms}/{\sqrt{n+1}}}$ contained in $H$. Note that
$\mathcal{C}^{+}(Q_{0},Ms)\subset B(Q_{0},Ms)$.

\begin{lemma}\label{lem2.7}
Given $\varepsilon>0$ and $L\in\mathcal{L}(\lambda,\Lambda, \alpha)$ there exists $M_{0}=M_{0}(n,\varepsilon,\alpha)>1$,
so that if $M\geq M_{0}$, and if $\omega$ denotes the $L-$elliptic
measure of $\mathcal{C}^{+}(Q_{0},Ms)$ as defined above, then for
$Q_{1}, Q_{2}\in \Delta(Q_{0},s)=\partial H\cap B(Q_{0},s)$, and $r_{1},
r_{2}\in (0,s]$
\begin{equation}
(1-\varepsilon)\bigg(\frac{r_{1}}{r_{2}}\bigg)^{n}\leq
\frac{\omega^{X}(\Delta(Q_{1},r_1))}{\omega^{X}(\Delta(
Q_{2},r_2))}\leq (1+\varepsilon)\bigg(\frac{r_{1}}{r_{2}}\bigg)^{n},
\end{equation}
as long as $X=(x,t)\in \partial\mathcal{C}^{+}(Q_{0},Ms/2)\cap
\mathcal{C}^{+}(Q_{0},Ms)$.
\end{lemma}

\begin{proof}
After rescaling we may assume without loss of generality that $Ms=1$,
$r_{i}\in (0,{1/M}]$, for $i=1,2$. First let us examine the case when
$X=(x,t)\in\partial\mathcal{C}^{+}(Q_{0},1/2)$, with $t\geq
{1/(2\kappa\sqrt{n+1}})$ and $\kappa>2$ to be chosen later. If
$\omega^X$ denotes the $L-$elliptic measure then
$$\omega^X(\Delta(Q_{1}, r_{1}))=\int_{\Delta(Q_1,r_1)}\langle A(Q)\nabla_Q G(Q,X),\nu \rangle d\sigma(Q)$$
or
\begin{equation}\label{repres1}
\omega^X(\Delta(Q_{1}, r_{1}))\leq\int_{\Delta(Q_1,r_1)}|\langle A(Q)\nabla_Q G(Q,X),\nu\rangle-\langle A(Q_0)\nabla_Q G(Q_0,X),\nu\rangle|d\sigma(Q)
\end{equation}
$$+\langle A(Q_0)\nabla_Q G(Q_0,X),\nu\rangle r_1^n$$
where $\nu$ denotes the inward unit normal to $H$ at $Q\in \partial H$. By the Hopf maximum principle (see \cite{ck}, \cite{hs}) there exists a constant $C_k=C_k(n,\lambda,\Lambda,\kappa)>0$ such that
$$\langle A(Q_0)\nabla_Q G(Q_0,X),\nu\rangle\geq C_k>0.$$
Moreover from the $C^{1,\alpha}$ regularity up to the boundary (\cite{gt}), we estimate the first term of (\ref{repres1}) to obtain
$$\omega^X(\Delta(Q_{1}, r_{1}))\leq (1+Cr_1^{\alpha})r_1^n\langle A(Q_0)\nabla_Q G(Q_0,X),\nu\rangle$$
where $C=C(n,\lambda,\Lambda,\kappa)$.
In a similar way, using the appropriate representation we have
$$\omega^X(\Delta(Q_{2}, r_{2}))\geq (1-Cr_2^{\alpha})r_2^n\langle A(Q_0)\nabla_Q G(Q_0,X),\nu\rangle$$
provided that $\Delta(Q_i,r_i)\subset \Delta(Q_0,2/M)$. Since $r_1,r_2<1/M$, we conclude that
$$\frac{\omega^{X}(\Delta(Q_{1},r_{1}))}{\omega^{X}(\Delta(Q_{2},r_{2}))} \leq \bigg(\frac{1+C/M^\alpha}{1-C/M^{\alpha}}\bigg) \bigg(\frac{r_1}{r_2}\bigg)^{n}$$
and
$$\frac{\omega^{X}(\Delta(Q_{1},r_{1}))}{\omega^{X}(\Delta(Q_{2},r_{2}))} \geq \bigg(\frac{1-C/M^{\alpha}}{1+C/M^{\alpha}}\bigg) \bigg(\frac{r_1}{r_2}\bigg)^{n}$$
provided that $M$ is large enough.

Now if $X=(x,t)\in\partial\mathcal{C}^{+}(Q_{0},{1/2})$, and
 $t\leq 1/{2\kappa\sqrt{n+1}}$, ${\omega^{X}(\Delta(Q_{1},r_{1}))}$ and ${\omega^{X}(\Delta(Q_{2},r_{2}))}$ vanish on
$B((x,0),1/4\sqrt{n+1}) \cap \partial H$ and are non negative in
$\mathcal{C}^{+}(Q_{0},1)$. Applying Theorem \ref{Thm2.3} we have
that for $\kappa>8$,
$$
\bigg|\frac{\omega^{X}(\Delta(Q_1,r_{1}))}{\omega^{X}(\Delta(Q_{2},r_{2}))}-
\frac{\omega^{(x,{1/{2\kappa\sqrt{n+1}}})}(\Delta(Q_{1},r_{1}))}{\omega^{(x,{1/{2\kappa\sqrt{n+1}}})}(\Delta(Q_2,r_{2}))}\bigg|\leq C\frac{\omega^{(x,{{1}/
{16\sqrt{n+1}}})}(\Delta(Q_{1},r_1))}{\omega^{(x,{{1}/
{16\sqrt{n+1}}})}(\Delta(Q_{2},r_2))} \bigg(\frac{1}
{\kappa}\bigg)^{\alpha}.
$$
On the other hand our new reference points $(x,1/
2\kappa\sqrt{n+1})$ and $(x,1/16\sqrt{n+1})$ fall into the first
case as described above, since $\kappa>8$. Thus
\begin{equation*}
\bigg(1-C\bigg(\frac{1}{\kappa}\bigg)^{\alpha}\bigg)
\bigg(\frac{1-C/M^{\alpha}}{1+C/M^{\alpha}}\bigg) \bigg(\frac{r_1}{r_2}\bigg)^{n}\leq
\frac{\omega^{X}(\Delta(Q_{1}, r_{1}))}{\omega^{X}(\Delta(Q_{2},
r_{2}))}\leq
\bigg(1+C\bigg(\frac{1}{\kappa}\bigg)^{\alpha}\bigg)
\bigg(\frac{1+C/M^{\alpha}}{1-C/M^{\alpha}}\bigg) \bigg(\frac{r_1}{r_2}\bigg)^{n}.
\end{equation*}
To finish the proof, for a given $\varepsilon>0$ choose $\kappa>8$
such that $(1-C(1/\kappa)^{\alpha})\geq \sqrt{1-\varepsilon}$, and
$(1+C(1/ \kappa)^{\alpha})\leq \sqrt{1+\varepsilon}$. Next for
that selection of $\kappa$ choose $M_{0}>2$ large enough so that
for $M\geq M_{0}$, $$\sqrt{1-\varepsilon}\leq \frac{1-C/M^{\alpha}}{1+C/M^{\alpha}}\leq \frac{1+C/M^{\alpha}}{1-C/M^{\alpha}}\leq \sqrt{1+\varepsilon}.$$
\end{proof}

\section{Optimal Doubling on Reifenberg Flat Domains.}\label{optimalDoubling}

As seen in Lemma \ref{lem2.4}, for $L\in\mathcal{L}(\lambda,\Lambda,\alpha)$, the $L-$elliptic measure of an NTA domain is a doubling
measure on $\partial\Omega$. In the present section we prove that for such $L$, the $L-$elliptic measure
of a Reifenberg flat domain $\Omega\subset \R^{n+1}$ with vanishing constant,
behaves asymptotically like the Euclidean measure in $\R^{n}$. Using the terminology introduced by
M. Korey (see \cite{ko}) we say that the $L-$elliptic measure of a Reifenberg flat domain with vanishing
constant is asymptotically optimally doubling.

\begin{thm}\label{Thm4.1}
Given $\varepsilon>0$, for $L\in\mathcal{L}(\lambda,\Lambda,\alpha)$, there exists $M(n,\varepsilon,\alpha)>1$, so that if
$M\geq M(n,\varepsilon,\alpha)$ there exists $\delta(\varepsilon,\alpha, M,r/s)=\delta>0$ such that,
for $0<r\leq s\leq {R/M}$ and any $(\delta,R)$-Reifenberg flat domain
$\Omega\subset \R^{n+1}$ we have:
$$
(1-\varepsilon)\bigg(\frac{r}{s}\bigg)^{n}\leq
\frac{\omega^{X}(B(Q_1,r))}{\omega^{X}(B(Q_2,s))}\leq
(1+\varepsilon)\bigg(\frac{r}{s}\bigg)^{n},
$$
where $Q_{1}, Q_{2}\in \partial\Omega\cap B(Q_0,s)$ for some
$Q_{0}\in\partial\Omega$, $X\in\Omega$, and $|X-Q_{0}|\geq Ms$.
\end{thm}

\begin{cor}\label{cor4.1}
Let $\Omega\subset\R^{n+1}$ be a Reifenberg flat
domain with vanishing constant, and $L\in\mathcal{L}(\lambda,\Lambda,\alpha)$. Then for any $X\in \Omega$ and $\tau\in (0,1)$
$$
\lim_{\rho\rightarrow 0} \inf_{Q\in \partial\Omega}
\frac{\omega^{X}(\partial\Omega\cap B(Q,\tau\rho))}{\omega^{X}(\partial\Omega\cap B(Q,\rho))}=
\lim_{\rho\rightarrow 0} \sup_{Q\in \partial\Omega}
\frac{\omega^{X}(\partial\Omega\cap B(Q,\tau\rho))}{\omega^{X}(\partial\Omega\cap B(Q,\rho))}= \tau^{n}.
$$
\end{cor}

The main idea of the proof is to compare the
elliptic measure of a Reifenberg flat domain with the elliptic measure of an
appropriate cylinder. In order to do this we need to introduce some extra notation.

Let $\Omega\subset \R^{n+1}$ be a $(\delta,R)$-Reifenberg flat domain, with
$\delta\leq \delta_{0}$. Let
$M>1$ be a large number to be determined later,
let $s>0$ be so that $Ms\leq R$. There exists an
$n$-dimensional plane $\mathcal{L}(Q_{0},Ms)$ containing $Q_{0}$ and such that
$$
\frac{1}{Ms}D[\partial\Omega\cap B(Q_{0},Ms),\mathcal{L}(Q_{0},Ms)\cap B(Q_{0},Ms)]\leq \delta,
$$
$$
T^{+}(Q_{0}, Ms)\subset\Omega\ \ \ {\rm{and}}\ \ \ T^{-}(Q_{0}, Ms)\subset \Omega^{c}.
$$
In particular if we define for $r=Ms$ or $r=Ms/2$
$$
\widetilde{\Omega}(Q_{0},Ms)=\Omega\cap
\{(x,t)\in\R^{n+1}: x\in \mathcal{L}(Q_{0}, Ms),\ |x-Q_{0}|\leq \frac{Ms}{\sqrt{n+1}},\
|t|\leq \frac{Ms}{\sqrt{n+1}}\},
$$
$$
\mathcal{C}^{+}(Q_{0},r)
=\{(x,t)\in\R^{n+1}: x\in \mathcal{L}(Q_{0},Ms),\
|x-Q_{0}|\leq \frac{r}{\sqrt{n+1}},\
2\delta r\leq t\leq \frac{r}{\sqrt{n+1}}\},
$$
$$
\mathcal{C}^{-}(Q_{0},r)
=\{(x,t)\in\R^{n+1}: x\in  \mathcal{L}(Q_{0}, Ms),\
|x-Q_{0}|\leq \frac{r}{\sqrt{n+1}},\
-2\delta r\leq t\leq \frac{r}{\sqrt{n+1}}\}
$$
and
$$
\mathcal{C}(Q_{0},Ms/2)
=\{(x,t)\in\R^{n+1}: x\in \mathcal{L}(Q_{0},Ms),\ |x-Q_{0}|
\leq \frac{Ms}{2\sqrt{n+1}}
,\ |t|\leq \frac{Ms}{2\sqrt{n+1}}\},
$$
then
$$
\mathcal{C}^{+}(Q_{0}, Ms)\subset\widetilde{\Omega}(Q_{0},Ms)\subset \mathcal{C}^{-}(Q_{0},Ms),
\ \ \ {\rm{and}} \ \ \
\mathcal{C}^{+}(Q_{0},Ms)\subset T^{+}(Q_{0},Ms).
$$
Note that the Hausdorff distance between $\mathcal{C}^{+}(Q_{0},Ms)$ and
$\mathcal{C}^{-}(Q_{0},Ms)$ is $4\delta Ms$. Besides if $n_{Ms,Q_{0}}$ denotes
the unit normal to $\mathcal{L}(Q_{0},Ms)$ chosen with the appropriate orientation,
then
$$
A(Q_{0},Ms)=Q_{0}+\frac{Ms}{4\sqrt{n+1}}n_{Ms,Q_{0}}\in
\mathcal{C}^{+}(Q_{0},Ms/2),
$$
for $\delta$ small enough
$$
B\bigg(A(Q_0,Ms),\frac{Ms}{8\sqrt{n+1}}\bigg)\subset \mathcal{C}^{+}(Q_{0},Ms/2)
$$
and
$$\dist\bigg[B\bigg(A(Q_0,Ms),\frac{Ms}{8\sqrt{n+1}}\bigg),
\partial\mathcal{C}^{+}(Q_{0},Ms/2)\bigg]\geq \frac{Ms}{16\sqrt{n+1}}.$$
\begin{remark}\label{rem4.2}
If $\Pi$ denotes the orthogonal projection from $\R^{n+1}$ onto
$\mathcal{L}(Q_{0},Ms)$ then
$$\Pi(\widetilde{\Omega}(Q_{0},Ms))=\{x\in \mathcal{L}(Q_{0},Ms): |x-Q_{0}|\leq
\frac{Ms}{\sqrt{n+1}}\}.$$
\end{remark}
Next we introduce the sets which arise from the intersection of
$\partial\widetilde{\Omega}(Ms, Q_{0})$ and cylinders having direction
$n_{Ms,Q_{0}}$. We denote by
$$
\Gamma(Q_{0},s)=\{(x,t)\in \R^{n+1}: x\in \mathcal{L}(Q_{0},Ms),\
|x-Q_{0}|\leq s,\ |t|\leq \frac{Ms}{\sqrt{n+1}}\}\cap
\partial\widetilde{\Omega}(Q_{0},Ms)
$$
and by
$$
\Gamma(Q,r)=\{(x,t)\in \R^{n+1}: x\in \mathcal{L}(Q_{0},Ms),\ |x-\Pi(Q)|\leq
r,\ |t|\leq \frac{Mr}{\sqrt{n+1}}\}\cap\partial\widetilde{\Omega}(Q_{0},Ms),
$$
for $Q\in\Gamma(Q_{0},s)$ and $r>0$ small enough so that
$\Gamma(Q,r)\subset\Gamma(Q_{0},s)$. In particular if $r=\tau s$
for some $\tau\in (4\delta M, 1)$, then
\begin{equation}\label{4.1}
\Gamma\bigg(Q,r{\sqrt{1-\bigg(\frac{2\delta M}{\tau}\bigg)^{2}}}\bigg)\subset\partial\Omega\cap B(Q,r)\subset\Gamma(Q,r).
\end{equation}
If $\tau$ is relatively large with respect to $2\delta M$, the
projections of these 3 sets on $L(Q_{0},Ms)$ have almost the same
area. In fact recall that $|\Pi(\Gamma(Q,r))|=\omega_{n}r^{n}$.

Let us denote by $\tilde{\omega}$ the elliptic measure of
$\widetilde{\Omega} (Ms, Q_{0})$ and by $\omega_{\pm}$ the
elliptic measures of $\mathcal{C}^{\pm}(Q_{0},Ms)$.

\begin{lemma}\label{lem4.1}
Given $\varepsilon>0$, for $L\in\mathcal{L}(\lambda,\Lambda,\alpha)$, there exists $M(n,\varepsilon,\alpha)>1$ such that
if $M\geq M(n,\varepsilon,\alpha)$ there exists
$\delta(\varepsilon,\alpha, M,r/s)=\delta>0$ so that if $\Omega$ is a
$(\delta, R)$-Reifenberg flat domain, then for $0<r\leq s\leq
{R/M}$, $Q_{0},Q\in\partial\Omega$, $B(Q,r)\subset B(Q_0,s)$, and $X\in\partial\mathcal{C}(Q_{0},Ms/2)\cap
\mathcal{C}^{+}(Q_{0},Ms)$,
\begin{equation}\label{4.2}
(1-\varepsilon)\omega^{X}_{+}(\Delta_{+}(Q_{+},r))\leq
\tilde{\omega}^{X}(\partial\Omega\cap B(Q,r)) \leq
(1+\varepsilon)\omega^{X}_{-}(\Delta_{-}(Q_{-},r))
\end{equation}
where $Q_{\pm}=\Pi(Q)\pm 2\delta Ms n_{Ms, Q_{0}}$, and
$\Delta_{\pm}(Q_{\pm},r)=B(Q_{\pm},r)\cap
\partial\mathcal{C}^{\pm}(Q_{0},Ms)$.
\end{lemma}
\begin{proof}
The basic idea is to compare the appropriate solutions of $Lu=0$ in $\mathcal{C}^{\pm}(Ms, Q_{0})$ and
$\widetilde{\Omega}(Ms, Q_{0})$ in order to apply the maximum
principle. Since Lemmata \ref{lem2.1} and \ref{lem2.7} are
valid we may adopt the proof of Lemma 4.1 in \cite{kt1}.
\end{proof}

The following lemma gives the opposite estimate when the pole is far away
from the boundary.

\begin{lemma}\label{lem4.2}
Given $\varepsilon>0$, for $L\in\mathcal{L}(\lambda,\Lambda,\alpha)$, there exists $M(n,\varepsilon,\alpha)>1$, so that
if $M\geq M(n,\varepsilon,\alpha)$ for $\kappa>2$ there exists
$\delta(\varepsilon,\alpha, M,\kappa,r/s)=\delta>0$ such that if $\Omega$
is a $(\delta, R)$-Reifenberg flat domain, then for $0<r\leq s\leq
{R/M}$, $Q_{0}\in\partial\Omega$, $Q\in\partial\Omega$, $B(Q,r)\subset B(Q_0,s)$, and $X=(x,t)\in\partial\mathcal{C}(Q_0,Ms/2)\cap\mathcal{C}^{+}(Q_0,Ms)$, with $t\geq
Ms/\kappa\sqrt{n+1}$,
\begin{equation}\label{4.14}
(1-\varepsilon)\omega^{X}_{-}(\Delta_{-}(Q_{-},r))\leq
\tilde{\omega}^{X}(B(Q,r)) \leq
(1+\varepsilon)\omega^{X}_{+}(\Delta_{+}(r, Q_{+})).
\end{equation}
where $Q_{\pm}=\Pi(Q)\pm 2\delta Ms n_{Ms, Q_{0}}$, and
$\Delta_{\pm}(Q_{\pm},r)=B(Q_{\pm},r)\cap
\partial\mathcal{C}^{\pm}(Q_{0},Ms)$.
\end{lemma}
\begin{proof}
Let $\varepsilon'=\varepsilon'(\varepsilon)>0$ to be chosen later.
We first prove that, for $M>2$ there exists
$0<\delta(\varepsilon,M,\kappa,r/s)$ so that
\begin{equation}\label{4.15}
\omega^{X}_{-}(\Delta_{-}(Q_{-},r))\leq
\frac{1+\varepsilon'}{1-\varepsilon'} \omega^{X}_{+}(\Delta_{+}(
Q_{+},r)).
\end{equation}
Let us first show how to obtain (\ref{4.14}) from (\ref{4.15}).
Choose $M$ large as in Lemma \ref{lem4.1}. Denote
$\delta':=\delta(\varepsilon',M, {r/s})$ the constant in that
lemma. Then for $\delta\leq
\min\{\delta',\delta(\varepsilon',M,\kappa,r/s)\}$, inequality
(\ref{4.2}) holds with $\varepsilon'$ instead of $\varepsilon$.
Combining (\ref{4.2}) and (\ref{4.15}) we obtain
$$
\frac{(1-\varepsilon')^{2}}
{1+\varepsilon'}\omega^{X}_{-}(\Delta_{-}(Q_{-},r))\leq
\tilde{\omega}^{X}(B(Q,r))\leq
\frac{(1+\varepsilon')^{2}}{1-\varepsilon'}
\omega^{X}_{+}(\Delta_{+}(Q_{+},r)).
$$
Choosing $\varepsilon'>0$ so that $1-\varepsilon\leq
(1-\varepsilon')^{2}/ (1+\varepsilon')$ and
$(1+\varepsilon')^{2}/(1-\varepsilon')\leq 1+\varepsilon$ we
obtain inequality (\ref{4.14}).

Now we continue with the proof of (\ref{4.15}). Recall that
$\mathcal{C}^{+}(Q_{0},Ms)\subset \widetilde{\Omega}(Q_{0},Ms)
\subset \mathcal{C}^{-}(Q_{0},Ms)$. Assume that
$\delta\leq\delta'$ and define
$$
u_{1}(x,t)=\omega_{-}^{(x,t)}(\Delta_{-}(Q_{-},r)) \ \ {\rm{for}}\ \
(x,t)\in \mathcal{C}^{-}(Q_{0},Ms),
$$
and
$$
u_{2}(x,t)=\omega_{+}^{(x,t)}(\Delta_{+}(Q_{+},r)) \ \ {\rm{for}}\ \
(x,t)\in \mathcal{C}^{+}(Q_{0},Ms).
$$
We compare $u_{1}(x, t-4\delta Ms)$ and $u_{2}(x,t)$ for
$(x,t)\in\partial \mathcal{C}^{+}(Q_{0},Ms)$. First note that if
$t=2\delta Ms$ or $|x-Q_{0}|= Ms/\sqrt{n+1}$ then $u_{1}(x,
t-4\delta Ms)= u_{2}(x,t)$. Indeed, if $t=2\delta Ms$ then
$u_{1}(x, t-4\delta Ms)$ vanishes for
$|x-\Pi(Q)|\geq r$ and it is equal to one
otherwise. The function $u_2$ has the same behavior. When
$|x-Q_0|=Ms/\sqrt{n+1}>>r$ both functions vanish.

Since $\mathcal{C}^{-}(Q_{0},Ms)$ is an NTA domain, $u_{1}$ is
non negative on $\mathcal{C}^{-}(Q_{0},Ms)$ and $u_{1}(x,t)=0$
for $$(x,t)\in\partial\mathcal{C}^{-}(Q_{0},Ms)\cap \{(x,t)\in
\R^{n+1}: t>\frac{3Ms}{4\sqrt{n+1}}\},$$ we apply Lemma
\ref{lem2.1} to get
$$
u_{1}(x, \frac{Ms}{\sqrt{n+1}}-4\delta Ms)\leq K_{1}\delta^{\beta}
$$
where $K_1$ depends on the NTA constants of
$\mathcal{C}^-(Q_0,Ms)$. Now consider a bounded function $v(x,t)$ such that
$$\begin{cases}
Lv(x,t)=0,& in \ \ \mathcal{C}^+(Q_0,Ms) \\
v(x,t)=0,& on\ \ t=2\delta Ms \\
v(x,t)=1,& on \ \ t=Ms/\sqrt{n+1} \\
v(x,t)\geq 0,& on \ \ |x-Q_0|=Ms/\sqrt{n+1}.
\end{cases}$$
Therefore for $(x,t)\in
\partial\mathcal{C}^{+}(Q_{0},Ms)$
\begin{equation}\label{4.16}
u_{1}(x, t-4\delta Ms)\leq u_{2}(x,t) + K_{1}v(x,t)\delta^\beta .
\end{equation}
By the maximum principle inequality (\ref{4.16}) holds for all
$(x,t)\in \mathcal{C}^{+}(Q_{0},Ms)$. Let
$$
\mathcal{R}= \mathcal{C}^{+}(Q_{0},Ms)\cap \bigg\{
(x,t)\in\R^{n+1}: x\in \mathcal{L}(Q_{0},Ms), |x-Q_{0}|\leq
\frac{Ms}{\sqrt{n+1}}\bigg(1-{1/{\kappa}}\bigg),$$ $$
\frac{Ms}{\kappa\sqrt{n+1}}\leq t\leq \frac{Ms}{\sqrt{n+1}}
\bigg(1-\frac{1}{\kappa}\bigg)\bigg\}.
$$
From the Hopf maximum principle and Harnack's inequality for
$\delta$ small enough and for $Y\in {\mathcal R}$
$$
u_{2}(Y)\ge K_{2}=K_{2}(n,\lambda, \Lambda, \kappa,\tau, M).
$$
Choosing $\delta>0$ even smaller we have $K_1
v(x,t)\delta^\beta\leq K_1C\delta^\beta \leq\varepsilon'K_2$, therefore for $(x,t)\in
\mathcal{R}$ we obtain
\begin{equation}\label{4.19}
u_{1}(x, t-4\delta Ms)\leq (1+\varepsilon')u_{2}(x,t).
\end{equation}
Applying classical interior estimates (see \cite{gt}, chapter 8) we conclude that there exists $\beta>0$ such that for
$\delta>0$ small enough and for $(x,t)\in \mathcal{R}$,
\begin{equation}\label{4.20}
u_{1}(x, t-4\delta Ms)\geq (1-C(\delta\kappa)^\beta)u_{1}(x,t)\geq
(1-\varepsilon')u_{1}(x,t).
\end{equation}
Combining (\ref{4.19}) and (\ref{4.20}) we have
$$
u_{1}(x,t)\leq \frac{1+\varepsilon'}{1-\varepsilon'}u_{2}(x,t)
$$
for $(x,t)\in \mathcal{R}$ and the proof is concluded.
\end{proof}

The next theorem is an immediate consequence of Lemmata
\ref{lem4.1} and \ref{lem4.2}. We refer the reader to Theorem 4.2
of \cite{kt1} for the details of the proof.
\begin{thm}\label{Thm4.2}
Given $\varepsilon>0$ for $L\in\mathcal{L}(\lambda,\Lambda,\alpha)$, there exists $M(n,\varepsilon,\alpha)>1$, so that
if $M\geq M(n,\varepsilon,\alpha)$ there exists
$\delta(\varepsilon,M,r/s)=\delta>0$ such that for any
$(\delta,R)$-Reifenberg flat domain $\Omega\subset \R^{n+1}$, and
$0<r\leq s\leq R/M$ we have
$$
(1-\varepsilon)\bigg(\frac{r}{s}\bigg)^{n}\leq
\frac{\tilde{\omega}^{X}(B(Q_1,r))}{\tilde{\omega}^{X}(B(Q_2,s))}\leq
(1+\varepsilon)\bigg(\frac{r}{s}\bigg)^{n},
$$
where $Q_{1}, Q_{2}\in\partial\Omega\cap B(Q_0,s)$ for some
$Q_{0}\in\partial\Omega$, $X\in\widetilde{\Omega}(Q_{0},Ms)\backslash
\mathcal{C}(Q_0,Ms/2)$. Here $\tilde{\omega}$ denotes the
elliptic measure of $\widetilde{\Omega}(Ms,Q_{0})$.
\end{thm}

We now show that as long as $X\in\widetilde{\Omega}(Q_{0},Ms)$ is far
away from $Q_{0}$, $\tilde{\omega}^{X}(E)/ \tilde{\omega}^{X}(E')$
and $\omega^{X}(E)/\omega^{X}(E')$ are comparable, whenever $E,
E'\subset \partial\Omega\cap B(Q,2s)$.
\begin{lemma}\label{lem4.3}
Given $\varepsilon >0$, $\tau\in (0,1)$ for $L\in\mathcal{L}(\lambda,\Lambda,\alpha)$, there exists $M(n, \varepsilon ,\alpha)> 0$ such that for
$M \geq M(n, \varepsilon,\alpha )$ there exists $\delta=\delta(\varepsilon, M,
\tau) >0$ such that if $\Omega$ is $(\delta, R)$-Reifenberg flat domain, $0< r
< s \leq R/M$, $Q_0 \in \partial \Omega$, $Q \in \partial
\Omega$, $B(Q,r) \subset B(Q_0,S)$ and $X \in
\widetilde{\Omega}(Q_0, Ms)\setminus \mathcal{C}(Q_0, Ms/2),$ then
$$ (1+\varepsilon)^{-1}\lim_{Y \rightarrow Q}\frac{\widetilde{G}(X,Y)}{G(X,Y)}
\leq \frac{d\widetilde{\omega}^X(Q)}{d \omega^X(Q)} \leq
(1+\varepsilon)\lim_{Y \rightarrow
Q}\frac{\widetilde{G}(X,Y)}{G(X,Y)}.$$
Here $\omega$ (resp. $\widetilde{\omega}$) denotes the elliptic
measure of $\Omega$ (resp. $\widetilde{\Omega}(Q_{0},Ms)$) with pole
at $X$, and $G$ (resp. $\widetilde{G}$) denotes the Green's functions
of $\Omega$ (resp. $\widetilde{\Omega}(Q_{0},Ms)$).
\end{lemma}
\begin{proof}
The Lebesgue differentiation theorem for Radon measures, ensures
that, for $\omega$-almost every $Q\in \partial\Omega\cap B(Q_0,2s)$
\begin{equation}\label{4.28}
\frac{d\widetilde{\omega}^{X}}{d{\omega^{X}}}(Q)= \lim_{r\rightarrow
0}\frac{\widetilde{\omega}^{X}(B(Q,r))}{\omega^{X}(B(Q,r))}.
\end{equation}
Consider a smooth function $\Psi_r$ such that $\Psi_r=1$ on $B(Q,r)$, ${\rm{spt}}(\Psi_r)\subset B(Q,2r)$, $|\nabla \Psi_r| \leq
\frac{C}{r}$ and $|D^2 \Psi_r| \leq \frac{C}{r^2}$. Let
$u_{r}$ satisfy $L u_{r}=0$ in $\Omega$ and $u_{r}=\Psi_{r}$ on
$\partial \Omega$. Let $\tilde{u}_{r}$ satisfy $L \widetilde{u}_{r}=0$
in $\widetilde{\Omega}(Ms,Q_{0})$ and $\widetilde{u}_{r}=\Psi_{r}$ on
$\partial\widetilde{\Omega}(Q_{0},Ms)$. Then
$$ u_r(X)= \int_{\partial \Omega}\Psi_r (Q)d\omega^X(Q)= -
\int_{\Omega} \langle A(Y)\nabla G(X,Y), \nabla \Psi_r\rangle dY$$
$$\widetilde{u}_r(X)= \int_{\partial \widetilde{\Omega}(Q_0, Ms)}\Psi_r (Q)d\widetilde{\omega}^X(Q)= -
\int_{\widetilde{\Omega}(Q_0, Ms)} \langle A(Y)\nabla
\widetilde{G}(X,Y), \nabla \Psi_r\rangle dY.$$

An argument similar to the one used to prove the Lebesgue differentiation theorem ensures that
\begin{equation}\label{tt1}
\frac{d\widetilde{\omega}^{X}}{d{\omega^{X}}}(Q)=\lim_{r\rightarrow
0}\frac{\tilde u_r(X)}{u_r(X)}.
\end{equation}

Let $A(Q)=A_Q$, then
$$\int_{\Omega} \langle A(Y)\nabla G(X,Y), \nabla \Psi_r\rangle dY =
\int_{\Omega} \langle \nabla G(X,Y), A_Q\nabla \Psi_r\rangle dY
+\int_{\Omega} \langle (A(Y)-A_Q)\nabla G(X,Y), \nabla
\Psi_r\rangle dY.$$
We estimate the last term by appealing (\ref{modulus}), H\"{o}lder's inequality, a boundary Cacciopoli estimate (see \cite{k1}, Lemma 1.21) and Lemma \ref{lem2.2},
\begin{eqnarray*}
\bigg|\int_{\Omega} \langle (A(Y)-A_Q)\nabla G(X,Y), \nabla \Psi_r\rangle dY \bigg|
&\leq & Cr^{\alpha} \bigg(\int_{B(Q,2r)}|\nabla
G|^2dY\bigg)^{1/2}\bigg(\int_{B(Q,2r)}|\nabla \Psi_r|^2dY\bigg)^{1/2} \\
&\leq &Cr^{\alpha}\frac{1}{r}r^{n/2}r^{n/2}\bigg({\mint}_{B(Q,2r)}|\nabla G|^2dY\bigg)^{1/2}\\
& \leq & Cr^{\alpha}r^{n-2}\bigg({\mint}_{B(Q,4r)}G^2dY\bigg)^{1/2}\\
& \leq & Cr^{n-2+\alpha}G(A(Q,r),X).\end{eqnarray*}

To estimate the first term on the right hand side note that
$$ \int_{\Omega} \langle \nabla G(X,Y), A_Q\nabla \Psi_r\rangle dY =
\int_{\Omega} {\rm{div}}(GA_Q\nabla \Psi_r)dY - \int_{\Omega} G
{\rm{div}}(A_Q\nabla \Psi_r)dY$$
thus
\begin{equation}\label{ur}
\bigg|u_r(X) - \int_{\Omega} G(X,Y) {\rm{div}}(A_Q\nabla \Psi_r)dY\bigg|\leq Cr^{n-2+\alpha}G(A(Q,r))+ \bigg|\int_{\Omega} {\rm{div}}(G(X,Y)A_Q\nabla \Psi_r)dY\bigg|.
\end{equation}
Note also that
$$ \bigg|\int_{\Omega}{\rm{div}}(G(X,Y)A_Q\nabla \Psi_r)dY\bigg|= \bigg|\int_{\widetilde{\Omega}(Q_0, Ms)} {\rm{div}}(G(X,Y)A_Q,\nabla \Psi_r)dY\bigg|.$$
If $\mathcal{F}=\widetilde{\Omega}(Q_0,Ms)\setminus C^{+}(Q_0,Ms)$ then
$$\int_{\widetilde{\Omega}(Q_0, Ms)}{\rm{div}}(G(X,Y)A_Q\nabla \Psi_r)dY=\int_{\mathcal{C}^{+}(Q_0,Ms)}{\rm{div}}(G(X,Y)A_Q\nabla \Psi_r)dY+
                                                  \int_{\mathcal{F}}{\rm{div}}(G(X,Y)A_Q\nabla \Psi_r)dY$$
and
\begin{eqnarray*}
\bigg|\int_{\mathcal{C}^{+}(Q_0,Ms)}{\rm{div}}(G(X,Y)A_Q\nabla \Psi_r)dY\bigg|
&=& \bigg|\int_{\{t= \frac{Ms\delta}{\sqrt{n+1}}\}} \langle G(X,Y)A_Q\nabla \Psi_r,e_n \rangle dS\bigg|\\
&\leq &C\sup_{Y\in B(Q_,2r)\cap \{t= \frac{Ms\delta}{\sqrt{n+1}}\}}\frac{1}{r}r^{n-1}G(X,Y)\\
&\leq &C\bigg(\frac{Ms\delta}{r}\bigg)^{\beta} G(A(Q,r),X)r^{n-2}.
\end{eqnarray*}
Now
\begin{eqnarray*}
\bigg|\int_{\mathcal{F}} G(X,Y) {\rm{div}}(A_Q\nabla \Psi_r)dY\bigg|
&\leq &C\sup_{\mathcal{F}\cap B(Q,2r)}G(X,Y) \cdot \sigma(\mathcal{F}\cap B(Q,2r))\frac{1}{r^2}\\
&\leq &C\bigg(\frac{Ms\delta}{r}\bigg)^{\beta}G(A(Q,r),X)\frac{1}{r^2}r^{n-1}Ms\delta\\
&= &C\bigg(\frac{Ms\delta}{r}\bigg)^{\beta+1}G(A(Q,r),X)r^{n-2}.
\end{eqnarray*}
In a similar way,
\begin{eqnarray*}
\bigg|\int_{\mathcal{F}}\langle A_Q\nabla \Psi_r, \nabla G \rangle dY\bigg|
&\leq &C\frac{1}{r}\int_{\mathcal{F}\cap B(Q,2r)}|\nabla G|dY \\
&\leq &C\frac{1}{r}\bigg(\int_{B(Q,2r)}|\nabla G|^2dY\bigg)^{1/2}\bigg(\mathcal{H}^n(\mathcal{F}\cap B(Q,2r))\bigg)^{1/2}\\
&\leq &C\frac{r^{n/2}}{r^2}G(A(Q,r),X)(Ms\delta r^{n-1})^{1/2}\\
&=& Cr^{n-2}G(A(Q,r),X)\bigg(\frac{Ms\delta}{r}\bigg)^{1/2}.
\end{eqnarray*}
Therefore
$$ \int_{\widetilde{\Omega}(Q_0,Ms)}{\rm{div}}(G(X,Y) A_Q\nabla \Psi_r)dY \leq C\bigg(\frac{Ms\delta}{r}\bigg)^{\beta}G(A(Q,r),X)r^{n-2}$$
for $\eta = \min\{\beta, \frac{1}{2}\}$. We use this estimate in (\ref{ur}),
$$\bigg|u_r(X)- \int_{\Omega}G(X,Y) {\rm{div}}(A_Q\nabla \Psi_r)dY\bigg|
\leq Cr^{n+\alpha-2}G(A(Q,r),
X)+C\bigg(\frac{Ms\delta}{r}\bigg)^{\eta}G(A(Q,r), X)r^{n-2}.$$
Note that a similar estimate holds for $\widetilde{u}_r(X)$ in terms of $\widetilde{G}$. Next we write
$$ \int_{\widetilde{\Omega}(Q_0,Ms)}\widetilde{G}(X,Y){\rm{div}}(A_Q\nabla \Psi_r)dY=
\int_{\widetilde{\Omega}(Q_0,Ms)}\frac{\widetilde{G}(X,Y)}{G(X,Y)}G(X,Y){\rm{div}}(A_Q\nabla
\Psi_r)dY $$
$$ = \int_{\widetilde{\Omega}(Q_0,Ms)}\bigg(\frac{\widetilde{G}(X,Y)}{G(X,Y)}-l(Q)\bigg)G(X,Y){\rm{div}}(A_Q\nabla
\Psi_r)dY+l(Q)\int_{\widetilde{\Omega}(Q_0,Ms)}G(X,Y) {\rm{div}}(A_Q\nabla
\Psi_r)dY$$ where $$l(Q)= \lim_{Y \rightarrow
Q}\frac{\widetilde{G}(X,Y)}{G(X,Y)}.$$

 We now choose $\tau s \leq r <s$ and
 $$ \frac{1}{M^{\gamma}} < \varepsilon'$$ where $\varepsilon'=\varepsilon'(\varepsilon)$. Then we choose $\delta$ such that
 $$ \bigg(\frac{\delta M}{\tau}\bigg)^{\eta}<
 \varepsilon'.$$ Combining the estimates above with Theorem \ref{Thm2.3} we have
 \begin{eqnarray}\label{tt2}
 \bigg|\widetilde{u}_r(X) &- &l(Q)u_r(X)\bigg| \leq \bigg|\widetilde{u}_r(X)-\int_{\Omega}\widetilde{G}(X,Y) {\rm{div}}(A_Q\nabla \Psi_r)dY\bigg|\\[3mm]
 &+&\bigg|\int_{\Omega}\widetilde{G}(X,Y) {\rm{div}}(A_Q\nabla \Psi_r)dY-l(Q)\int_{\Omega}G(X,Y) {\rm{div}}(A_Q\nabla \Psi_r)dY\bigg|\nonumber\\[3mm]
 &+ &l(Q)\bigg|u_r(X)-\int_{\Omega}G(X,Y) {\rm{div}}(A_Q\nabla \Psi_r)dY\bigg|\nonumber\\
 &\lesssim& r^{n+\alpha-2}\widetilde{G}(A(Q,r),
X)+\varepsilon'r^{n-2}\widetilde{G}(A(Q,r), X)\nonumber\\[3mm]
&+&l(Q)r^{n+\alpha-2}G(A(Q,r),
X)+\varepsilon'l(Q)r^{n-2}G(A(Q,r),
X)\nonumber\\[3mm]
&+&\bigg(\frac{r}{Ms}\bigg)^{\gamma}l(Q)\bigg(\frac{\delta Ms}{r}\bigg)^{\beta+1}r^{n-2}G(A(Q,r),
X)\nonumber\\[3mm]
&\lesssim & r^{n-2}G(A(Q,r),X)\bigg\{ (r^\alpha+\varepsilon')\frac{\widetilde{G}(A(Q,r),X)}{G(A(Q,r),X)}+ (r^{\alpha}+\varepsilon')l(Q)\bigg\}\nonumber\\[3mm]
&\lesssim & u_r(X)\bigg\{ (r^\alpha+\varepsilon')\frac{\widetilde{G}(A(Q,r),X)}{G(A(Q,r),X)}+ (r^{\alpha}+\varepsilon')l(Q)\bigg\}\nonumber
\end{eqnarray}
since by the maximum principle
$$r^{n-2}G(A(Q,r),X)\lesssim \omega^X(B(Q,r))\lesssim u_r(X).$$
Furthermore since
$$\lim_{r\rightarrow 0}\frac{\widetilde{G}(A(Q,r),X)}{G(A(Q,r),X)}= l(Q)\ \ \hbox{ then }\  \
\bigg|\frac{\widetilde{u}_r(X)}{u_r(X)}-l(Q)\bigg| \lesssim \varepsilon l(Q).$$
We conclude the proof by combining (\ref{tt1}), (\ref{tt2}) and choosing $\varepsilon'$ in terms of $\varepsilon$.
\end{proof}
\begin{cor}\label{cor4.2}
Given $\varepsilon>0$, for $L\in\mathcal{L}(\lambda,\Lambda,\alpha)$, there exists $M(n,\varepsilon,
\alpha)>1$ so that
if $M\ge M(n,\varepsilon, \alpha)$ there exists
$\delta(n,\alpha, \varepsilon,M)>0$, such that if $\Omega\subset \R^{n+1}$
is a $(\delta, R)$-Reifenberg flat domain with $\delta\in (0,
\delta(n,\varepsilon)]$, $Q_{0}\in\partial \Omega$, $0<s\leq
{R/M}$, $E, E'\subset \partial\Omega\cap B(Q_0,2s)$, and $X\in
\widetilde{\Omega}(Q_{0},Ms)\backslash\mathcal{C}(Q_0,Ms/2)$ then
$$
(1-\varepsilon)\frac{\tilde{\omega}^{X}(E)}{\tilde{\omega}^{X}(E')}\leq
\frac{\omega^{X}(E)}{\omega^{X}(E')}\leq
(1+\varepsilon)\frac{\tilde{\omega}^{X}(E)}{\tilde{\omega}^{X}(E')}.$$
\end{cor}

\begin{proof} We choose $\varepsilon'$ to depend on $\varepsilon$ such that Lemma \ref{lem4.3} is satisfied. From Theorem \ref{Thm2.3} we have
$$
|\ell(Q_{0})-\ell(Q)|\leq C \frac{\widetilde{G}(X,A(
Q_{0},Ms))}{G(X,A(Q_{0},Ms))}
\bigg(\frac{|Q-Q_{0}|}{Ms}\bigg)^{\gamma}
$$
for $Q\in\partial\Omega\cap B(Q_0,2s)$. In addition Lemma \ref{lem2.5}
guarantees that there exists a constant $C>1$ so that
$$
C^{-1}\ell(Q_{0})\leq \frac{\widetilde{G}(X,A(
Q_{0},Ms))}{G(X,A(Q_{0},Ms))}\leq C\ell(Q_{0}).
$$
Hence
$$
\bigg(1-\frac{C}{M^{\gamma}}\bigg)\ell(Q_{0})\leq \ell(Q)\leq
\bigg(1+\frac{C}{M^{\gamma}}\bigg)\ell(Q_{0}).
$$
Since
$$
\widetilde{\omega}^{X}(E)=\int_{E}\frac{d\widetilde{\omega}^{X}}{d\omega^{X}}(Q)d\omega^{X}(Q)\leq
(1+\varepsilon')\int_{E}\ell(Q)d\omega^{X}(Q)\leq (1+\varepsilon')\bigg(1+\frac{C}{M^\gamma}\bigg)l(Q_0)\omega^X(E)
$$
and
$$\widetilde{\omega}^X(E)\geq (1+\varepsilon')^{-1}\bigg(1-\frac{C}{M^\gamma}\bigg)l(Q_0)\omega^X(E)$$
we have that
$$
(1+\varepsilon')^{-2}\frac{1-(C/M^{\gamma})}{1+(C/M^{\gamma})}\cdot
\frac{\omega^{X}(E)}{\omega^{X}(E')}\leq
\frac{\tilde{\omega}^{X}(E)}{\tilde{\omega}^{X}(E')}\leq(1+\varepsilon')^{2}\frac{1+(C/M^{\gamma})}{1-(C/M^{\gamma})}\cdot
\frac{\omega^{X}(E)}{\omega^{X}(E')}.
$$
Choosing $M$ and $\varepsilon'$ appropriately we conclude the proof of Corollary \ref{cor4.2}.
\end{proof}

\textbf{Proof of Theorem \ref{Thm4.1}:} For
$\varepsilon'=\varepsilon'(\varepsilon)$, let
$M(n,\varepsilon',\alpha)>1$ be as in Theorem \ref{Thm4.2} and Corollary
\ref{cor4.2}. For $M\geq M(n,\varepsilon',\alpha)$ there exists
$\delta(n,\varepsilon', \alpha, {r/s},M)>0$ so that if $\Omega$ is a
$(\delta,R)$ Reifenberg flat domain with $\delta\leq
\delta(n,\varepsilon', {r/s},M)$ then Theorem \ref{Thm4.2}, Lemma \ref{lem4.3} and
Corollary \ref{cor4.2} hold. Namely for $0<r\leq s\leq {R/M}$ we
have that
$$
(1-\varepsilon')\bigg(\frac{r}{s}\bigg)^{n}\leq
\frac{\widetilde{\omega}^{X}(B(Q_1,r))}{\widetilde{\omega}^{X}(B(Q_2,r))}\leq
(1+\varepsilon')\bigg(\frac{r}{s}\bigg)^{n},
$$
where $Q_{1}, Q_{2}\in\partial\Omega\cap B(Q_0,s)$ for some $Q_{0}\in\partial
\Omega$, $X\in\widetilde{\Omega}(Q_{0},Ms)\backslash\mathcal{C}(Q_0,Ms/2)$. Moreover
$$
(1-\varepsilon')\frac{\widetilde{\omega}^{X}(B(Q_1,r))}{\widetilde{\omega}^{X}(B(Q_2,s))}\leq
\frac{\omega^{X}(B(Q_1,r))}{\omega^{X}(B(Q_2,s))}\leq (1+\varepsilon')\frac{\widetilde{\omega}^{X}(B(Q_1,r))}{\widetilde{\omega}^{X}(B(Q_2,s))}.
$$
Therefore for $\varepsilon'>0$ so that $1-\varepsilon\leq
(1-\varepsilon')^{2}$ and $(1+\varepsilon')^{2}\leq
1+\varepsilon$, and $X\in\Omega\cap
\partial B(Q_{0},Ms/2)$
$$
(1-\varepsilon)\bigg(\frac{r}{s}\bigg)^{n}\omega^{X}(B(Q_2,s))\leq \omega^{X}(B(Q_1,r)) \leq
(1+\varepsilon)\bigg(\frac{r}{s}\bigg)^{n}\omega^{X}(B(Q_2,s)).
$$
The maximum principle guarantees that for all $X\in\Omega
\backslash B(Ms/2, Q_{0})$
$$
(1-\varepsilon)\bigg(\frac{r}{s}\bigg)^{n}\leq
\frac{\omega^{X}(B(Q_1,r))}{\omega^{X}(B(Q_2,s))}\leq (1+\varepsilon)\bigg(\frac{r}{s}\bigg)^{n}.
$$

\section{Regularity on chord arc domains.}\label{chord}

In this section we prove that on a chord arc domain with small enough constant if $L$ is either in $\mathcal L(\lambda, \Lambda, \alpha)$ or if it is a perturbation of the
Laplacian then the elliptic measure is an $A_\infty$ weight with respect to surface measure.
In the case that $\Omega $ is a chord arc domain with vanishing constant and  $L\in \mathcal L(\lambda, \Lambda, \alpha)$ we show that the logarithm of elliptic kernel (i.e. the density of the elliptic measure with respect to the surface measure)
is in VMO. A key step in these proofs is Semmes' Decomposition for
chord arc domains with small constant (see \cite{kt1}, Theorem
2.2).

Let $\Omega$ be a $(\delta,R)$-CAD for $\delta$ small enough so Theorem 2.2 in $\cite{kt1}$ holds. Let $P\in\partial \Omega$ and let $r>0$ small enough so the construction in Lemma 5.1 in $\cite{kt1}$ goes through. In this case there exist two Lipschitz functions $h^+$ and $h^-$ defined in $\mathcal{L}(P,r)$ such that $h^-\leq h^+$ and $\| \nabla h^{\pm}\| _{\infty}\leq \eta $ where $\eta \simeq \delta^{1/4}$. Let
$$\Omega^{+}=\{(x,t)\in \R^{n+1}\ :\ x\in \mathcal{L}(P,r), \ t>h^{+}(x) \}$$
and
$$\Omega^{-}=\{(x,t)\in \R^{n+1}\ :\ x\in \mathcal{L}(P,r), \ t>h^{-}(x) \}.$$
As in Lemma 5.1 in \cite{kt1} the graphs $\Gamma^{\pm}$ of $h^{\pm}$ approximate $\partial \Omega$ in $\mathcal{C}(P,r)$ from above and below respectively, in the sense that
\begin{equation}\label{heart}
D\bigg[\Gamma^{\pm}\cap B(P,r); \partial \Omega \cap B(P,r)\bigg]\leq \eta r \ \ {\rm{and}} \ \
\sigma( \Gamma^+ \cap \Gamma^-\cap B(P,r))\geq \bigg(1-c_1\exp\{-c_2/\eta\}\bigg)\omega_nr^n
\end{equation}
where $c_1, c_2$ are positive constants as in (\cite{kt1}, Theorem 2.2). Moreover
$$\Omega^+\cap \mathcal{C}(P,r)\subset \Omega\cap \mathcal{C}(P,r)\subset \Omega^-\cap \mathcal{C}(P,r).$$

\begin{lemma}\label{AinftyHolder}
Let $L\in \mathcal L(\lambda,\Lambda,\alpha)$.
There exists $\delta(n)>0$ such that if $\Omega \subset \R^{n+1}$ is a $(\delta,R)-$CAD with $0<\delta\leq\delta(n)$ and $X\in \Omega$ then $\omega^X\in A_\infty(d\sigma)$
where $\sigma=\mathcal{H}^n \res \partial \Omega$.
\end{lemma}

\begin{proof}
Choose $\delta(n)>0$ such that Semmes decomposition applies $\Omega$ as in Lemma 5.1 in \cite{kt1}. For $X\in
\Omega$ let $d=\dist(X,\partial \Omega)$ and $\omega^X=\omega$.
Let $0<r\le \min\{R/2, d/4\}$. For  $P\in \partial\Omega$
let $\Delta=\partial \Omega \cap B(P,r)$ and $A=A(P,r)$ be the non-tangential interior point of
$\Omega\cap\mathcal{C}(P,2r)$. We may assume that
$\Omega^+\cap \mathcal{C}(P,2r)\subset \Omega\cap \mathcal{C}(P,2r)$.
We denote by $\omega^+$ be the $L-$elliptic measure of
$\Omega^+\cap\mathcal{C}(P,2r)$.
Since for $E\subset\Delta$, $\omega^A (E)\simeq \omega(E)/\omega(\Delta)$.
 It is enough to prove that for
$\alpha'\in (0,1)$ small, there exists $\beta\in (0,1)$ so that if
$\omega^A(E)<\alpha'$ then $\sigma(E)/\sigma(\Delta)<\beta$.

Assume that $\omega^A(E)<\alpha'$. We decompose $E$ as in Lemma 5.2 of \cite{kt1}, $E=E_1\cup E_2$ where $E_1=E\cap \partial\Omega^+$ and $E_2=E\setminus \partial \Omega^+$. By the maximum principle $\omega_+^A(E_1)\leq \omega^A(E)<\alpha'$. We write
\begin{equation}\label{Holdersemmes}
\frac{\sigma(E)}{\sigma(\Delta)}=\frac{\sigma(E_1)}{\sigma(\Delta)}+\frac{\sigma(E_2)}{\sigma(\Delta)}.
\end{equation}
Since $\Omega^+$ is a Lipschitz domain, $\omega_+\in A_\infty(d\sigma_+)$
so there are positive constants $\theta$, $C_1$, $C_2$ such that
$$C_1\bigg(\frac{\sigma_+(E_1)}{\sigma_+(\Delta_+)}\bigg)^{1/\theta}\leq \frac{\omega_+(E_1)}{\omega_+(\Delta_+)}\leq C_2\bigg(\frac{\sigma_+(E_1)}{\sigma_+(\Delta_+)}\bigg)^\theta$$
where $\Delta_+=\partial \Omega\cap B(\Pi(P),h^+(\Pi(P)),r\sqrt{1+\eta^2})$ and $\sigma_+$ denotes the surface measure of $\partial \Omega^+$. Therefore the first term of (\ref{Holdersemmes}) is estimated by
$$\frac{\sigma(E_1)}{\sigma(\Delta)} \leq \frac{\sigma_+(E_1)}{\sigma_+(\Delta_+)}\cdot\frac{\sigma_+(\Delta_+)}{\sigma(\Delta)}\lesssim {\alpha'}^\theta(1+\eta^2)^{(n+1)/2}(1+\eta).$$
Finally the second term of (\ref{Holdersemmes}) is controlled using the Semmes' Decomposition estimate for
chord arc domains with small constant (see \cite{kt1}, Theorem
2.2). That is,
$$\frac{\sigma(E_2)}{\sigma(\Delta)}\leq \frac{c_1\exp(-c_2/\eta)\omega_nr^n}{\sigma(\Delta)}\lesssim (1+\eta)\exp(-c_2/\eta).$$
Gathering all the estimates and choosing $\alpha'>0$ and $\delta>0$ small enough, since $\eta\simeq \delta^{1/4}$ we conclude that $\sigma(E)/\sigma(\Delta)<\beta<1$.
\end{proof}
An immediate consequence of Lemma \ref{AinftyHolder} is the following Corollary.

\begin{cor}\label{cor5.1}
Let $L\in \mathcal L(\lambda,\Lambda,\alpha)$. There exist $\delta(n,\lambda,\Lambda)=\delta_0>0$, $\mu>0$ and $\beta>0$
such that if $\Omega\subset \R^{n+1}$ is a $(\delta, R)$-CAD with $0<\delta\leq \delta_0$,
for $X\in \Omega$, $\Delta=\partial \Omega \cap B(Q,s)$ with $Q\in\partial\Omega$,
$s\leq \min\{{{\rm{dist}(X,\partial \Omega)}/4}, {R/4}\}$,
and $E\subset \Delta$ is a measurable set then
$$
\frac{\omega^{X}(E)}{\omega^{X}(\Delta)}\leq c\bigg(\frac{\sigma(E)}{\sigma(\Delta)}\bigg)
^{2\mu}.
$$
Moreover if $k_{X}={{d\omega^{X}}/{d\sigma}}$ then
$$
\bigg(\frac{1}{\sigma(\Delta)}\int_{\Delta} k_{X}^{1+2\beta} d\sigma\bigg)^{1/{1+2\beta}}
\leq c \frac{1}{\sigma(\Delta)}\int_{\Delta} k_{X} d\sigma
$$
where $c>1$ denotes a constant that depends only on $n, \lambda, \Lambda$.
\end{cor}

The next theorem states that the density satisfies a reverse
H\"{o}lder inequality with optimal constant. The proof is very
similar to the proof of Theorem 5.2 in \cite{kt1}, which we can
adopt in our case due to the $C^{0,\alpha}$ regularity of the
coefficients. Here we present only the main steps of the
proof.

\begin{thm}\label{Thm5.1} Let $L\in \mathcal L(\lambda,\Lambda,\alpha)$.
Given $\varepsilon>0$, and $N>0$ there exists
$\delta_0=\delta(\varepsilon,N,\lambda,\Lambda,n)>0$ such that if $\Omega\subset
\R^{n+1}$ is a $(\delta, R)$-CAD with $\delta\in (0,\delta_0)$ there exists
$\gamma=\gamma(\varepsilon, N, \lambda,\Lambda, \alpha)>0$, so that for any surface ball
$B\subset\partial\Omega$ with radius $s \leq\gamma/2$, if
$X\in\Omega$ with $\dist(X,\partial\Omega)\geq N$, and
$k_{X}(Q)=k(Q)= \frac{d\omega^{X}}{d\sigma}(Q)$ then
$$
\bigg(\frac{1}{\sigma(B)}\int_{B}k^{1+\beta}d\sigma\bigg)^{1/{1+\beta}}\leq
(1+\varepsilon)\frac{1}{\sigma(B)}\int_{B}k d\sigma,
$$
where $\beta>0$.
\end{thm}
\begin{proof}
We intend to
apply Semmes decomposition in the set $\Delta (Q_0,r)=B(Q_0,r)\cap
\partial\Omega$ with $r=Ms$ for $M>>1$ large enough.

Let $X\in
\partial\mathcal{C}(Q_{0},r/2)\cap\Omega^{+} \subset\widetilde{
\Omega}(Q_{0},Ms)$ and $\widetilde{\Omega}(Q_{0},Ms)=\Omega\cap
\mathcal{C}(Q_{0},Ms)$. We denote by $\omega$ the elliptic
measure of $\Omega$ with pole at $X$, by $\tilde{\omega}$ the
elliptic measure of $\widetilde{\Omega}(Q_{0},Ms)$ with pole at
$X$, by $\tilde{\omega}_{-}$ the elliptic measure of
$\Omega^{-}\cap \mathcal{C}(Q_{0},Ms)$ with pole $X$ and by
$\omega_{-}$ the elliptic measure of $\Omega^{-}$ with pole
$X$. Moreover we denote by $k_{-}(Q)=d\omega_{-}/d \sigma_{-}$,
$\tilde{k}_{-}(Q)=d\tilde{\omega}_{-}/d \sigma_{-}$.

We need to estimate
\begin{equation}\label{5.1*}
\int_{\Delta}k^{1+\beta}d\sigma = \int_{\Delta\backslash
\partial\Omega^{-}} k^{1+\beta}d\sigma + \int_{\Delta\cap \partial\Omega^{-}}k^{1+\beta}d\sigma.
\end{equation}
To estimate the first term of (\ref{5.1*}), we apply Semmes decomposition to get
$$
\sigma(\Delta\backslash\partial\Omega^{-})\leq
C_{1}\exp(-\frac{C_{2}}{\eta})\omega_{n}(Ms)^{n}
$$
or
$$
\frac{\sigma(\Delta\backslash
\partial\Omega^{-})}{\sigma(\Delta)}\leq 2C_{1}
\exp(-\frac{C_{2}}{\eta})M^{n}.
$$
where $\eta\simeq \delta^{1/4}$.
Applying Corollary \ref{cor5.1} and choosing
$\delta>0$ small enough we conclude that
\begin{eqnarray*}
\frac{1}{\sigma(\Delta)}\int_{\Delta\backslash\partial\Omega^{-}}
k^{1+\beta}d\sigma & \leq &
\frac{1}{\sigma(\Delta)}\bigg(\int_{\Delta\backslash\partial\Omega^{-}}k^{1+2\beta}
d\sigma\bigg)^{\frac{1+\beta}{1+2\beta}}\sigma(\Delta\backslash\partial\Omega^{-})
^{\frac{\beta}{1+2\beta}} \\
& \leq &
K\bigg(\frac{\sigma(\Delta\backslash\partial\Omega^{-})}{\sigma(\Delta)}\bigg)
^{\frac{\beta}{1+2\beta}}\bigg(\frac{1}{\sigma(\Delta)}\int_{\Delta}k^{1+\beta}
d\sigma\bigg)^{1+\beta}\\
&\leq &
K\bigg(2C_{1}\exp(-\frac{C_{2}}{\delta})M^{n}\bigg)^{\frac{\beta}{1+2\beta}}
\bigg(\frac{1}{\sigma(\Delta)}\int_{\Delta}k^{1+\beta}d\sigma\bigg)^{1+\beta}
\\
& \leq & \varepsilon'
\bigg(\frac{1}{\sigma(\Delta)}\int_{\Delta}k^{1+\beta}
d\sigma\bigg)^{1+\beta}.
\end{eqnarray*}
In order to estimate the second term of (\ref{5.1*}), we need to show that
\begin{equation}\label{5.1main}
k_(Q)\leq (1+\varepsilon')^6k_-(Q)\frac{\omega^X(\Delta)}{\omega_-^X(\Delta_-)}
\end{equation}
for every $X\in\partial \mathcal{C}(Ms/2,Q_0)\cap\Omega$, where $k$ and $k_-$ denote the elliptic kernel with pole
$X$, $\Delta_{-}=B((\Pi(Q_{0}), h^{-}(\Pi(Q_{0}))),s)\cap \partial\Omega^{-}$ and $Q\in\Delta\cap \Delta_-$.

The proof of (\ref{5.1main}) follows the same guidelines as the corresponding proof  in \cite{kt1}.
We include the proof in the case that the pole is far from the boundary in order to illustrate which results need
to be used in this case.  We refer the reader to \cite{kt1} for the proof of the case when the pole is close to the boundary.

Let $X=(x,t)$ with
$t\geq {{Ms}/{{\kappa}\sqrt{n+1}}}$. Let $G_{0}\subset\Delta\cap\partial\Omega^{-}$
be the set of density points of
$\Delta\cap \partial\Omega^{-}$.

By Lebesgue density theorem

$$
\int_{\Delta\cap \partial\Omega^{-}}k^{1+\beta}d\sigma=
\int_{G_{0}}k^{1+\beta}d\sigma
$$
and applying Corollary \ref{cor5.1} for $Q\in G_{0}$, we have

$$
\lim_{\Delta_{0}\downarrow Q}\frac{\omega(\Delta_{0}\cap\partial\Omega^{-})}{\omega(\Delta_{0})}=1,
$$
and
$$
k(Q)=\lim_{\Delta_{0}\downarrow  Q}\frac{\omega(\Delta_{0})}{\sigma(\Delta_{0})}=
\lim_{\Delta_{0}\downarrow Q}\frac{\omega(\Delta_{0}\cap\partial\Omega^{-})}{\sigma(\Delta_{0})}
$$
where $\Delta_{0}$ is a surface ball centered at $Q$ and contained in $\Delta$. Let $F=\Delta_{0}\cap\partial\Omega^{-}$ and apply the maximum principle to obtain
$$
\frac{\widetilde{\omega}(F)}{\widetilde{\omega}(\Delta)}
\leq \frac{\widetilde{\omega}_{-}(F)}{\widetilde{\omega}_{-}(\Delta_{-})}\cdot
\frac{\widetilde{\omega}_{-}(\Delta_{-})}{\widetilde{\omega}(\Delta)}
$$
where $$
1-\varepsilon'\leq\frac{\widetilde{\omega}_{-}(\Delta_{-})}{\widetilde{\omega}(\Delta)}\leq 1+\varepsilon'
$$
since Lemmata \ref{lem4.1} and \ref{lem4.2} are valid. Now using Corollary \ref{cor4.2} we obtain

$$
\frac{\omega(F)}{\omega(\Delta)}\leq
(1+\varepsilon')^{3} \frac{\omega_{-}(F)}{\omega_{-}(\Delta_{-})}
$$
and
$$
\frac{\omega(F)}{\sigma(\Delta_{0})}\leq
(1+\varepsilon')^{3}\frac{\omega_{-}(F)}{\sigma_{-}(\Delta_{0}^{-})}
\frac{\sigma_{-}(\Delta_{0}^{-})}{\sigma(\Delta_{0})}\frac{\omega(\Delta)}
{\omega_{-}(\Delta_{-})}
$$
where $\Delta_{0}^{-}$ is a surface ball in $\partial\Omega^{-}$ centered at $Q$ and with the same radius as $\Delta_{0}$.
Using the fact that $\partial\Omega^{-}$ is a Lipschitz graph with
small constant less
we conclude that for $\delta>0$ small enough,
$$
\frac{\omega(F)}{\sigma(\Delta_{0})}\leq
(1+\varepsilon')^{5}\frac{\omega_{-}(F)}{\sigma_{-}(\Delta_{0}^{-})}
\frac{\omega(\Delta)}{\omega_{-}(\Delta_{-})}
\leq (1+\varepsilon')^{5}\frac{\omega_{-}(\Delta_{0}^{-})}{\sigma_{-}(\Delta_{0}^{-})}
\frac{\omega(\Delta)}{\omega_{-}(\Delta_{-})}.
$$
Therefore letting $\Delta_{0}\downarrow Q$ we conclude that
$$
k(Q)\leq (1+\varepsilon')^{5}\frac{\omega(\Delta)}{\omega_{-}(\Delta_{-})} k_{-}(Q).
$$

The proof of the case when the pole is close to the boundary uses Theorem \ref{Thm2.3} and the ideas of the proof of Theorem 4.2 in \cite{kt1}.

Next we estimate the second term in (\ref{5.1*}). For $X\in\partial\mathcal{C}(Q_0,Ms/2)\cap \Omega$,
\begin{eqnarray*}
\frac{1}{\sigma(\Delta)}\int_{G_0}k^{1+\beta}d\sigma &\leq& \bigg[(1+\varepsilon')^{6}
\frac{\omega(\Delta)}{\omega_{-}(\Delta_{-})}\bigg]^{1+\beta}
\frac{1}{\sigma(\Delta)}\int_{G_0}k_{-}^{1+\beta}d\sigma\\
&\leq & (1+\varepsilon') \bigg[(1+\varepsilon')^{6}
\frac{\omega(\Delta)}{\omega_{-}(\Delta_{-})}\bigg]^{1+\beta}
\frac{1}{\sigma_{-}(\Delta_{-})}\int_{\Delta_{-}}k_{-}^{1+\beta}d\sigma_{-}\\
&\leq &(1+\varepsilon') \bigg[(1+\varepsilon')^{7}
\frac{\omega(\Delta)}{\omega_{-}(\Delta_{-})}\bigg]^{1+\beta}
\bigg(\frac{1}{\sigma_{-}(\Delta_{-})}\int_{\Delta_{-}}k_{-}d\sigma_{-}\bigg)^{1+\beta}\\
&\leq &(1+\varepsilon')(1+\varepsilon')^{8(1+\beta)}
\bigg(\frac{1}{\sigma(\Delta)}\int_{\Delta}k d\sigma\bigg)^{1+\beta}.
\end{eqnarray*}

Combining all the estimates above and choosing $\varepsilon'$ in term of $\varepsilon$

we have that
$$
\bigg(\frac{1}{\sigma(B)}\int_{B}k^{1+\beta}d\sigma\bigg)^{1/{1+\beta}}\leq
(1+\varepsilon)\frac{1}{\sigma(B)}\int_{B}k d\sigma,
$$
\end{proof}

The regularity result is a consequence of the following corollary.

\begin{cor}\label{cor5.2}
Let $L\in \mathcal L(\lambda,\Lambda,\alpha)$.
Given $\varepsilon>0$, and $N>0$ there exists
$\delta_0=\delta(\varepsilon,N,\lambda,\Lambda,n)>0$ such that if $\Omega\subset
\R^{n+1}$ is a $(\delta, R)$-CAD with $\delta\in(0,\delta_0)$ there exists
$\gamma=\gamma(\varepsilon,N,R,\lambda,\Lambda,w)>0$, so that for any surface ball
$B\subset\partial\Omega$ with radius $s \leq\gamma/2$, if
$X\in\Omega$ with $\dist(X,\partial\Omega)\geq N$, and
$k_{X}(Q)=k(Q)= \frac{d\omega^{X}}{d\sigma}(Q)$, then
$$
\frac{1}{\sigma(B)}\int_{B}|\log k-\bigg(\frac{1}{\sigma(B)}
\int_{B}\log k d\sigma\bigg)|d\sigma\leq \varepsilon.
$$
\end{cor}
\begin{proof}
We will use Sarason's lemma and John-Niremberg's in the following
manner. Let $\varepsilon'(\varepsilon)>0$ to be determined later.
For $\varepsilon'$ and $N$ let $\delta$ and $\gamma$ be as in
Theorem \ref{Thm5.1} and $d\nu=(\int_{\Delta}k d\sigma)^{-1}k
d\sigma$. From H\"older's inequality we have
$$
\int_{B}k^{1-\beta} d\sigma\leq
\bigg(\int_{B}k^{1+\beta}d\sigma\bigg)^{{1-\beta}/
{1+\beta}}\sigma(B)^{{2\beta}/{1+\beta}}.
$$
Hence for $\varepsilon'$ small enough
$$
\int_{B}k^{\beta}d\nu \int_{B}k^{-\beta}d\nu\leq
\bigg(\frac{1}{\sigma(B)}\int_{B} k d\sigma\bigg)^{-2}
\bigg(\frac{1}{\sigma(B)}\int_{B} k^{1+\beta}
d\sigma\bigg)^{2/{1+\beta}}\leq 1+3\varepsilon'.
$$
Applying now Sarason's lemma (see \cite{s}) together with
John-Nirenberg's inequality guarantees that for $p\in [1,\infty)$,
if $s\leq {\gamma/8}$,
$$
\bigg(\frac{1}{\omega(B)}\int_{B}|\log k -
c_{B}|^{p}d\omega\bigg) ^{1/p}\leq
\frac{C_{p}}{\beta}{\varepsilon'}^{1/3},
$$
where $c_{B} = 1/{\omega(B)}\int_{B}\log k d\omega$.
From the theory of $A_{\infty}$-weights we have for some $p$ large
enough
$$
\bigg(\int_{B}k^{{-1}/{p-1}}\bigg)^{p-1}\leq C
\frac{\sigma(B)^{p}}{\omega(B)}.
$$
Thus applying H\"{o}lder's inequality we have
$$
\frac{1}{\sigma(B)}\int_{B}|\log k -
c_{B}|d\sigma\leq
C\bigg(\frac{1}{\omega(B)}\int_{B}|\log k -
c_{B}|^{p}d\omega\bigg) ^{1/p}\leq
C(\beta,p){\varepsilon'}^{1/3}.
$$
Choosing $\varepsilon'$ so that
$C(\beta,p){\varepsilon'}^{1/3}\leq {\varepsilon/2}$ we conclude
that for a surface ball $B$ with radius $s\leq {\gamma/8}$
$$
\frac{1}{\sigma(B)}\int_{B}|\log k-\bigg(\frac{1}{\sigma(B)}
\int_{B}\log k d\sigma\bigg)|d\sigma\leq \varepsilon.
$$
\end{proof}

\begin{cor}\label{cor5.3}
Let $L\in \mathcal L(\lambda,\Lambda,\alpha)$.
Let $\Omega\subset\R^{n+1}$ be a chord arc domain with vanishing
constant. Then for any $X\in\Omega$, $\log k_{X}\in
VMO(\partial\Omega)$.
\end{cor}

We now concentrate in the case when $L$ is perturbation of the Laplacian, i.e we assume that (\ref{normfkp}) holds. The crucial step in the proof of Theorem \ref{mainpertu} is to compare the $L$-elliptic measures of the Lipschitz domains $\Omega^{\pm}$ we constructed above.

\begin{lemma}\label{mainDJK}
Let $L$ be a uniformly elliptic operator in divergence form
satisfying the assumptions of Theorem \ref{mainpertu}. Let $\Gamma^{\pm}$ be
defined as above. Then there exists $\theta>0$ such that for
$Q\in\partial \Omega$ and $s>0$ if $E\subset \Gamma^+\cap\Gamma^-\cap
B(Q,s)$
\begin{equation}\label{label4.1}
\frac{\omega_-(E)}{\omega_-(\Delta_-)}\leq C\bigg(\frac{\omega_+(E)}{\omega_+(\Delta_+)}\bigg)^\theta
\end{equation}
where $\Delta_{\pm}=B(Q^{\pm},s)$ with $Q^{\pm}=(\Pi(Q),h^{\pm}(\Pi(Q)))$ and $\Pi$ is the projection in $\mathcal{L}(P,r)$. Here $\omega^{\pm}$ denote the $L-$harmonic measures of $\Omega^{\pm}\cap \mathcal{C}(P,2r)$ with pole outside $B(P,r)$.
\end{lemma}
\begin{proof} The proof will follow the lines of Main Lemma in \cite{djk}. Let $G \subset \Gamma^+ \cap \Gamma^-\cap B(P,r)$ and denote by $Q^{\pm}= (q, h^{\pm}(q))$. If $X=(x, x_{n+1})$ then
$$\dist(X, \Gamma^{\pm} ) \simeq |x_{n+1} - h^{\pm}(x) |$$ and for $q$ such that
$h^{+}(q) \neq h^{-}(q)$
\begin{equation}\label{labelA}
h^{+}(q) - h^{-}(q) \simeq \dist(Q^{+}, \Gamma^{-}) \simeq \dist(Q^{-}, \Gamma^{+}).
\end{equation}
We proceed by constructing a Whitney decomposition of $\R^{n+1}\setminus G$. Extract a subfamily $\{Q_i^-\}$ such that $\Gamma^-\setminus G\subset \cup Q^{-}_i$. By (\ref{labelA}) note that $\dist(Q^{-}_i, \Gamma^{+})\simeq {\rm{diam}}\  Q^{-}_i$. Since ${\rm{Lip}}(h^+)\leq \eta$ and $\eta<<1$ there exist a family $\{Q_i^+\}$ obtained by vertical translations from $\{Q_i^-\}$ and such that $Q^-\in Q_i^-$ if and only if $Q^+\in Q^+_i$. Furthermore
$${\rm{diam}}\ Q^{-}_i ={\rm{diam}} \ Q^{+}_i \simeq \dist(Q^{+}_i, \Gamma^{-}) \simeq \dist(Q^{-}_i, \Gamma^{+})$$
and
$\Gamma^{+}\setminus G \subset \cup Q^{+}_i.$

We can find $Q^{*}_i$ such that $2Q^{-}_i \cup 2Q^{+}_i \subset Q^{*}_i$ and ${\rm{diam}}\ Q^{*}_i \simeq {\rm{diam}}\ Q^{\pm}_i.$
Next we define the measure $\mu$ for $F \subset \Gamma^{-} \cap \overline{B}(P, r)$ by
$$\mu(F)= \omega_{+}(F \cap G) + \sum_i \frac{\omega_{-}(F \cap Q^{-}_i)}{\omega_{-}(Q^{*}_i)}\omega_{+}(Q^{*}_i).$$
We will prove the following claim.

\textit{Claim.} If $Q\in\Gamma^-$ and $B(Q,s)\subset B(P,r)$ for $F \subset \Gamma^{-} \cap B(P, r)$ then
\begin{equation}\label{claim1}
\frac{\mu(F)}{\mu(B(Q,s))}\lesssim \frac{\omega_-(F)}{\omega_-(B(Q,s))}
\end{equation}
and
\begin{equation}\label{claim2}
\mu(B(P,r))\simeq 1.
\end{equation}
From the claim, using the real variable lemma of Coifman and Fefferman (see \cite{CF}) we conclude (\ref{label4.1}).

\textit{Proof of claim.} Let $Q \in \Gamma^{-} \cap \overline{B}(P, r)$ and $0<s \leq r.$ If for all $i$, $ Q^{-}_i \cap B(Q, s/2)= \emptyset$ then $B(Q,s/2)\subset G$, $Q\in \Gamma^+$ and by the doubling property of $\omega_+$
$$ \mu(B(Q, s))\geq \omega_{+}(B(Q, s/2) \cap G)=\omega_{+}(B(Q, s/2))\gtrsim \omega_{+}(B(Q,s)).$$ If there exists an $i$ such that
$Q^{-}_i \cap B(Q, s/2)\neq \emptyset$
by the doubling property of $\omega_+$
$$ \mu(B(Q, s))\gtrsim \omega_{+}(B(Q, s) \cap G)+\sum _{Q_i\cap B(Q, s/2) \neq \emptyset} \omega_{+}(Q^{*}_i).$$
Moreover, if $Q \in \Gamma^{+} \cap \Gamma ^{-},$ and $Q_i\cap B(Q, s/2) \neq \emptyset$ for some $i$,
$$\mu(B(Q, s))\gtrsim \omega_{+}(B(Q, s) \cap G)+\omega_{+}(B(Q, s/2)\setminus G)\gtrsim \omega_{+}(B(Q, s/2)).$$
Now if $Q \notin \Gamma^{+}$, $Q=Q^{-} \in Q^{-}_i$ (for the same $i$ as above) and $Q^{+} \in Q^{+}_i.$ In addition
$$ B(Q^{+}, s/4)\cap \Gamma^{+}\setminus G  \subset \cup 2Q^{+}_i\  {\rm{and}}\  Q^{-}_i\cap B(Q, s/2) \neq \emptyset$$
since for $(x, h^{+}(x)) \in B(Q^{+}, s/4) \cap \Gamma^{+}\setminus G$,
\begin{eqnarray*}
|(x, h^{+}(x))-(q, h^{-}(q))|
&\leq & |(x, h^{-}(x))-(x, h^{+}(x))|+|(x,h^{-}(x))-(q,h^{-}(q))|\\
&\lesssim & {\rm{diam}}\  Q_i^-+s/4+\eta s/4.
\end{eqnarray*}
Thus $\sum \omega_{+}(Q^{*}_i) \gtrsim \omega_{+}(B(Q^{+}, s/4)\setminus G) $ with $ Q^{-}_i \cap B(Q, s/2) \neq \emptyset$ and
$$\mu(B(Q,s)) \gtrsim \omega_{+}(B(Q,s)\cap G)+\omega_{+}(B(Q^{+},s/4)\setminus G)$$
for $Q\in\Gamma^-\setminus \Gamma^+$.

Now if $B(Q,s/(2\cdot 10^{6}))\cap G = \emptyset$, for all $X=(x,h^{+}(x))=(x,h^{-}(x))$ we have
$$|(q, h^{-}(q))-(x,h^{-}(x))|>s/(2\cdot 10^{6}).$$
Hence if $|q-x|<s/(2\cdot 10^{6})$ then
$$|x-q|\geq \frac{s}{2\cdot 10^{6}}-\frac{\eta s}{2\cdot 10^{6}}$$
and for $\eta$ small enough
$$|(q, h^{+}(q))-(x,h^{+}(x))|\geq |x-q|-\eta s/(2\cdot 10^{6})\geq \frac{s}{4\cdot 10^{6}}.$$
Thus by the doubling property of $\omega_+$,
$$ \mu(B(Q,s))\geq \omega_{+}(B(Q^{+},s/(4\cdot 10^{6})))\gtrsim  \omega_{+}(B(Q^{+},cs)\gtrsim  \omega_{+}(B(Q,cs)).$$
On the other hand, if $B(Q,s/(2\cdot 10^{6})) \cap G \neq \emptyset,$ then $B(Q^+,s/4)\subset B(Q,s)$ and again by the doubling property of $\omega_+$
$$\mu(B(Q,s))\geq \omega_{+}(B(Q^{+},s/4))\gtrsim \omega_{+}(B(Q,cs)).$$
Thus, in any case we have shown that
\begin{equation}\label{label**}
\mu(B(Q,s)) \gtrsim \omega_+(B(Q^+, s/4))\gtrsim \omega_{+}(B(Q,cs)).
\end{equation}

Let $Q \in \Gamma^{-}\cap B(P,r)$, $B(Q,s)\subset B(P,r)$ and consider two cases.

\underline{Case 1.} For every $i$, $Q^{-}_i \cap B(Q,s) \neq 0$ and $\rm{diam}\ Q^{-}_i \leq 100s.$

Then, $B(Q, Cs)\cap \Gamma^{+} \neq \emptyset.$ For simplicity, let $A= A^{+}(Q^{+}, Cs)=A^{-}(Q^{-}, Cs)$ the non-tangential points of $\Omega^{\pm}$ at $Q^{\pm}$ at radius $Cs$. Since $Q^{-}_i \cap B(Q,s) \neq 0,$
the distance of $Q=Q^{-}$ to $Q^{-}_i$ is less or equal to $s$ and $2Q^{-}_i \subset B(Q, Cs),$ so $Q^{*}_i \subset B(Q, Cs)$
and using the Carleson estimate in \cite{cfms} we have
\begin{equation}\label{labelDelta}
\frac{\omega^{A}_{\pm}(Q^{*}_i)}{\omega^{A}_{\pm}(B(Q,Cs))} \simeq \frac{\omega_{\pm}(Q^{*}_i)}{\omega_{\pm}(B(Q,Cs))}
\end{equation}
and
\begin{equation}\label{labelC}
\frac{\omega^{A}_{-}(F\cap Q^{-}_i)}{\omega_{-}^A(B(Q,Cs))} \simeq \frac{\omega_{-}(F\cap Q^{-}_i)}{\omega_{-}(B(Q,Cs))}.
\end{equation}
Similarly
\begin{equation}\label{labelD}
\frac{\omega^{A}_{+}(F\cap G)}{\omega_{+}^A(B(Q,Cs))} \simeq \frac{\omega_{+}(F\cap G)}{\omega_{+}(B(Q,Cs))}.
\end{equation}
Recall that $\omega_{\pm}^A(B(Q,Cs))\simeq 1$ thus (\ref{labelC}) and (\ref{labelD}) become
\begin{equation}\label{labelE}
\omega^{A}_{-}(F\cap Q^{-}_i)\simeq \frac{\omega_{-}(F\cap Q^{-}_i)}{\omega_{-}(B(Q,Cs))}
\end{equation}
and
\begin{equation}\label{labelDF}
\omega^{A}_{+}(F\cap G) \simeq \frac{\omega_{+}(F\cap G)}{\omega_{+}(B(Q,Cs))}.
\end{equation}
In addition, since $\omega_{+}$ is a doubling measure,
\begin{equation} \label{ena}
\omega^{A}_{+}(Q^{*}_i) \lesssim \omega^{A}_{+}(Q^{+}_i).
\end{equation}
If $Z \in Q^{+}_i,$ then $\omega^{Z}_{-}(Q^{*}_i) \simeq 1=\omega^{Z}_{+}(Q^{+}_i).$ For $Z \in \partial(\Omega^+\cap B(P,Cr))\setminus Q_i^+$, $\omega ^{Z}_-(Q^{*}_i)\geq 0$ and $\omega ^{Z}_{+}(Q^{+}_i)=0.$ Therefore by the maximum principle for $Z \in \partial(\Omega^+\cap B(P,Cr))$
\begin{equation}\label{dyo}
\omega ^{Z}_{+}(Q^{+}_i) \lesssim \omega ^{Z}_{-}(Q^{*}_i).
\end{equation}
From (\ref{ena}), (\ref{dyo}) we deduce that,
\begin{equation}\label{labelH}
\omega ^{A}_{+}(Q^{*}_i) \lesssim \omega ^{A}_{-}(Q^{*}_i).
\end{equation}
Thus, by (\ref{label**}), (\ref{labelDelta}), (\ref{labelC}), (\ref{labelE}) and (\ref{labelH}) we have
\begin{eqnarray*}
\frac{\mu(F)}{\mu(B(Q,s))} &\lesssim &
\frac{\omega_{-}(F\cap G)+\sum_i \frac{\omega_{-}(F\cap Q^{-}_i)}{\omega_{-}(Q^{*}_i)} \omega_{+}(Q^{*}_i)}{\omega_{+}(B(Q,Cs))}\\
&\lesssim &\omega^{A}_{+}(F\cap G)+\sum_i \frac{\omega^{A}_{-}(F\cap Q^{-}_i)}{\omega^{A}_{-}(Q^{*}_i)} \omega^{A}_{+}(Q^{*}_i)\\
&\lesssim &\omega^{A}_{-}(F\cap G)+\sum_i \omega^{A}_{-}(F\cap Q^{-}_i)\\
&\lesssim &\omega^{A}_{-}(F)\\
&\lesssim& \frac{\omega_{-}(F)}{\omega_{-}(B(Q,s))}.
\end{eqnarray*}
\underline{Case 2.} Suppose there exists $Q^{-}_i$ such that $Q^{-}_i \cap B(Q,s) \neq \emptyset$ and $\rm{diam}\  Q^{-}_i > 100s.$ This implies $B(Q,s)\cap G=\emptyset$ and $B(Q,s)\subset Q_i^*$.

If $Q^{-}_l \cap B(Q,s) \neq \emptyset,$
then, $\dist(Q^{-}_i, Q^{-}_l) \leq 2s,$ with ${\rm{diam}}\ Q^{-}_i \simeq {\rm{diam}}\ Q^{-}_l.$ Since $\omega_{-}$ and $\omega_{+}$
are doubling measures and $Q_i^*$, $Q_l^*$ have large overlaps, $\omega_{\pm}(Q^{*}_i) \simeq \omega_{\pm}(Q^{*}_l).$ Thus for $F\subset B(Q,s)$
\begin{eqnarray*}
\mu(F)&=& \sum_i \frac{\omega_{-}(F\cap Q^{-}_i)}{\omega_{-}(Q^{*}_i)} \omega_{+}(Q^{*}_i)\\
&\simeq &\frac{\omega_{+}(Q^{*}_l)}{\omega_{-}(Q^{*}_l)} \sum_{i}\omega_{-}(F \cap Q^{-}_i)\\
&\simeq &\frac{\omega_{+}(Q^{*}_l)}{\omega_{-}(Q^{*}_l)} \omega_{-}(F)
\end{eqnarray*}
and similarly
$$\mu(B(Q,s))\simeq \frac{\omega_+(Q_l^*)}{\omega_-(Q_l^*)}\omega(B(Q,s))$$
which yields
$$\frac{\mu(F)}{\mu(B(Q,s))} \simeq \frac{\omega_{-}(F)}{\omega_-(B(Q,s))}.$$
This concludes the proof of (\ref{claim1}) in the claim. To prove (\ref{claim2}) recall that by (\ref{heart})
$$\sigma(\Gamma^+ \cap \Gamma^-\cap B(P,r))\geq \bigg(1-c_1\exp\{-c_2/\eta\}\bigg)r^n\omega_n.$$
Hence by the doubling property of $\omega_+$
$$\mu(B(P,r))\geq \omega_+(B(P,r)\cap G)\gtrsim \omega_+(B(P,r))\simeq 1.$$
Clearly by the doubling character of $\omega_+$ we also have that $\mu(B(P,r))\lesssim 1$.
\end{proof}

\textit{Proof of Theorem \ref{mainpertu}.}

Choose $\delta(n)>0$ such that Semmes decomposition applies $\Omega$ as in Lemma 5.1 in \cite{kt1}. For $X\in
\Omega$ let $d=\dist(X,\partial \Omega)$ and $\omega^X=\omega$.
Let $0<r\le \min\{R/2. d/4\}$. For  $P\in \partial\Omega$
let $\Delta=\partial \Omega \cap B(P,r)$ and $A=A(P,r)$ be the non-tangential interior point of
$\Omega\cap\mathcal{C}(P,2r)$. We may assume that
$\Omega^+\cap \mathcal{C}(P,2r)\subset \Omega\cap \mathcal{C}(P,2r)$.
We denote by $\omega^+$ be the $L-$elliptic measure of
$\Omega^+\cap\mathcal{C}(P,2r)$.
Since for $E\subset\Delta$, $\omega^A (E)\simeq \omega(E)/\omega(\Delta)$
 It is enough to prove that for
$\alpha\in (0,1)$ small, there exists $\beta\in (0,1)$ so that if
$\omega^A(E)<\alpha$ then $\sigma(E)/\sigma(\Delta)<\beta$.

Assume that $\omega^A(E)<\alpha$. We decompose $E$ as in Lemma 5.2 of \cite{kt1}, $E=E_1\cup E_2$ where $E_1=E\cap \partial\Omega^+$ and $E_2=E\setminus \partial \Omega^+$. By the maximum principle $\omega_+^A(E_1)\leq \omega^A(E)<\alpha$. We write
\begin{equation}\label{Pertusemmes}
\frac{\sigma(E)}{\sigma(\Delta)}=\frac{\sigma(E_1)}{\sigma(\Delta)}+\frac{\sigma(E_2)}{\sigma(\Delta)}
\end{equation}

In this case we do not know if $\omega_+\in A_\infty(d\sigma_+)$. On the other hand since $\Omega^-$ is a Lipschitz domain and $L$
(extended to be the Laplacian in $\Omega^c$) is a perturbation of the Laplacian satisfying (\ref{normfkp}) we know that $\omega_-\in A_\infty(d\sigma_-)$ by Theorem \ref{fkp1991}.
Therefore there are positive constants $\gamma$, $C_1$, $C_2$ such that
\begin{equation}\label{tt3}
C_1\bigg(\frac{\sigma_-(E_1)}{\sigma_-(\Delta_-)}\bigg)^{1/\gamma}\leq \frac{\omega_-(E_1)}{\omega_-(\Delta_-)}\leq C_2\bigg(\frac{\sigma_-(E_1)}{\sigma_-(\Delta_-)}\bigg)^\gamma
\end{equation}
where $\Delta_-=\partial \Omega\cap B(\Pi(P),h^-(\Pi(P)),r\sqrt{1+\eta^2})$ and $\sigma_-$ denotes the surface measure of $\partial \Omega^-$. The first term of (\ref{Pertusemmes}) is estimated combining (\ref{tt3}) and Lemma \ref{mainDJK},
\begin{eqnarray}\label{ul-A}
\frac{\sigma(E_1)}{\sigma(\Delta)} &\leq& \frac{\sigma_-(E_1)}{\sigma_-(\Delta_-)}\frac{\sigma_-(\Delta_-)}{\sigma(\Delta)}\\
&\lesssim &\bigg(\frac{\omega_-(E_1)}{\omega_-(\Delta_-)}\bigg)^{\gamma}\frac{\sigma_-(\Delta_-)}{\sigma(\Delta)}\nonumber\\
&\lesssim & \bigg(\frac{\omega_+(E_1)}{\omega_+(\Delta_+)}\bigg)^{\theta'}\frac{\sigma_-(\Delta_-)}{\sigma(\Delta)}\nonumber\\
&\lesssim &\alpha^{\theta'}(1+\eta^2)^{n+1/2}(1+\eta)\nonumber
\end{eqnarray}
where $\theta'=\theta'(\theta,\gamma)$. Finally the second term of (\ref{Pertusemmes}) is controlled using the Semmes' Decomposition estimate for
chord arc domains with small constant (see \cite{kt1}, Theorem
2.2). That is,
\begin{equation}\label{ul-B}
\frac{\sigma(E_2)}{\sigma(\Delta)}\leq \frac{c_1\exp(-c_2/\eta)\omega_nr^n}{\sigma(\Delta)}\lesssim (1+\eta)\exp(-c_2/\eta).
\end{equation}
Combining (\ref{ul-A}) and (\ref{ul-B}), and choosing $\alpha>0$ and $\delta>0$ small enough(recall $\eta\simeq \delta^{1/4}$) we conclude that $\sigma(E)/\sigma(\Delta)<\beta<1.$ \qed

\end{document}